\documentclass[
    leqno,
    11pt,
]{amsart}
\usepackage{amsmath}
\usepackage{amssymb}
\usepackage{amsthm}
\usepackage{aliascnt}
\usepackage[dvipsnames]{xcolor}
\usepackage{dsfont}
\usepackage{mathrsfs}
\usepackage{mathtools}
\usepackage{float}
\usepackage{tikz}
\usepackage{pgfplots}
\usepackage{hyperref}
\usepackage{nameref}
\usepackage{xspace}
\usepackage[
    nameinlink
]{cleveref}
\usepackage{autonum}
\usepackage{hyphenat}
\usepackage{microtype}
\usepackage{xfrac}

\RequirePackage{etex}

\hypersetup{
    colorlinks=true,
    linkcolor=blue,
    citecolor=orange,
    filecolor=magenta,
    urlcolor=blue,
    pdftitle={Gradual smoothing},
    }
\let\oldtocsection=\tocsection
\let\oldtocsubsection=\tocsubsection
\renewcommand{\tocsection}[2]{\hspace{0em}\oldtocsection{#1}{#2}}
\renewcommand{\tocsubsection}[2]{\hspace{1em}\oldtocsubsection{#1}{#2}}
\newtheorem{Theorem}{Theorem}[section]

\newaliascnt{Lemma}{Theorem}
\newtheorem{Lemma}[Lemma]{Lemma}
\aliascntresetthe{Lemma}

\newaliascnt{Proposition}{Theorem}
\newtheorem{Proposition}[Proposition]{Proposition}
\aliascntresetthe{Proposition}

\newaliascnt{Corollary}{Theorem}
\newtheorem{Corollary}[Corollary]{Corollary}
\aliascntresetthe{Corollary}

\newaliascnt{Definition}{Theorem}
\newtheorem{Definition}[Definition]{Definition}
\aliascntresetthe{Definition}

\newaliascnt{Remark}{Theorem}
\newtheorem{Remark}[Remark]{Remark}
\aliascntresetthe{Remark}

\newaliascnt{Example}{Theorem}

\aliascntresetthe{Example}

\crefname{Theorem}{Theorem}{Theorems}
\Crefname{Theorem}{Theorem}{Theorems}

\crefname{Lemma}{Lemma}{Lemmas}
\Crefname{Lemma}{Lemma}{Lemmas}

\crefname{Proposition}{Proposition}{Propositions}
\Crefname{Proposition}{Proposition}{Propositions}

\crefname{Corollary}{Corollary}{Corollaries}
\Crefname{Corollary}{Corollary}{Corollaries}

\crefname{Definition}{Definition}{Definitions}
\Crefname{Definition}{Definition}{Definitions}

\crefname{Remark}{Remark}{Remarks}
\Crefname{Remark}{Remark}{Remarks}

\crefname{Example}{Example}{Examples}
\Crefname{Example}{Example}{Examples}

\calclayout

\newcommand{\RR}{\mathbb{R}}
\newcommand{\cont}{\mathcal{C}}

\newcommand{\ent}{\textup{Ent}}

\renewcommand{\theequation}{\arabic{section}.\arabic{equation}}

\numberwithin{equation}{section}

\begin{document}
\title[Gradual smoothing]{Gradual smoothing: strong hypercontractivity and logarithmic Sobolev inequalities}
\date{}
\author[A.\,de Pablo, D.\,Lee,  F.\,Quirós and J.\,Ruiz-Cases]{Arturo de Pablo, David Lee, Fernando Quirós \& Jorge Ruiz-Cases}

\address{A.\,de Pablo.
    Departamento de Matem\'{a}ticas, Universidad Carlos III de Madrid, 28911-Legan\'{e}s, Spain
    \& Instituto de Ciencias Matem\'aticas ICMAT (CSIC-UAM-UC3M-UCM), 28049-Madrid, Spain.}
\email{arturop@math.uc3m.es}

\address{D.\,Lee.
    School of Mathematical and Physical Sciences, University of Technology Sydney (UTS), Australia.}
\email{davidchanwoo.lee@uts.edu.au}

\address{F.\,Quir\'os.
    Departamento de Matem\'{a}ticas, Universidad Aut\'onoma de Madrid,
    \& Instituto de Ciencias Matem\'{a}ticas ICMAT (CSIC-UAM-UCM-UC3M),
    28049-Madrid, Spain.}
\email{fernando.quiros@uam.es}

\address{J.\,Ruiz-Cases.
    Departamento de Matem\'{a}ticas, Universidad Aut\'onoma de Madrid,
    \& Instituto de Ciencias Matem\'{a}ticas ICMAT (CSIC-UAM-UCM-UC3M),
    28049-Madrid, Spain.}
\email{jorge.ruizc@uam.es}
\urladdr{https://sites.google.com/view/jorgeruizcases}

\begin{abstract}
    We study the possibility of a gradual improvement as time progresses of the regularity of solutions to evolution problems of parabolic type driven by Lévy-type operators, not necessarily translation invariant. In the course of our analysis we study the equivalence between general smoothing effects and a family of logarithmic Sobolev inequalities. This equivalence allows us to identify a new type of regularization, \emph{strong hypercontractivity}, characterized by the existence of a time at which solutions belong to every~$L^p$ space with~$p$ finite. It can also be used to prove logarithmic Sobolev inequalities in a context not previously seen in the literature. We then show that any purely nonlocal Lévy-type operator whose kernel is comparable to that of~$\log(I-\Delta)$ is strongly hypercontractive, but fails to be supercontractive and, consequently, also fails to be ultracontractive. Furthermore, in the translation-invariant case, we also prove that solutions get bounded eventually and start improving in differentiability right after doing so. Finally, we show that this behaviour only appears if the kernel defining the operator behaves as $|x-y|^{-N}$ for small interactions (\emph{$0^+$-order operators}): more singular kernels yield instantaneous smoothing, while less singular ones do not produce any regularization.
\end{abstract}



\maketitle

\setcounter{tocdepth}{2}
\tableofcontents

\section{Introduction}

\subsection{Motivation and goal}
We investigate the possibility of a gradual improvement, as time evolves, of the regularity of solutions to
\begin{equation}\label{problem-L} \tag{$\textup{P}_{\mathcal{L}}$}
    \partial_tu +\mathcal{L}u=0\quad\text{in }\mathbb{R}^N\times(0,\infty),\qquad u(\cdot,0)=u_0\quad\text{in }\mathbb{R}^N,
\end{equation}
for linear operators~$\mathcal{L}$ as general as possible. By \emph{gradual smoothing} of solutions we mean that solutions belong to progressively better integrability spaces as time evolves (\emph{hypercontractivity}), and once they become bounded, their differentiability also starts improving gradually. We will study the whole improving process, although we focus more on the first part of it: the gradual improvement in integrability. As part of our analysis, we prove that a general (quantitative) notion of hypercontractivity is equivalent to a family of logarithmic Sobolev inequalities. This characterization leads to the introduction of a new regularization phenomenon, which we call \emph{strong hypercontractivity}. 

Under minimal contractivity and positivity-preserving assumptions on the solutions of~\eqref{problem-L}, no generality is lost by restricting to Lévy operators. On the other hand, it is natural to consider only operators such that solutions to the equation conserve their mass. Finally, we exclude both second-order local terms, since uniformly elliptic diffusion yields instantaneous regularization, and local transport terms, since they have no influence on the smoothing properties of the equation. Hence, the class in which we are going to look for gradual smoothing is that of purely nonlocal self-adjoint Lévy-type operators, which have the form
\begin{equation}\label{general-L}
    \mathcal{L} f(x) = \text{P.V.} \int_{\RR^N} (f(x) - f(y))J(x,y)\, dy
\end{equation}
for some measurable, nonnegative  kernel~$J$ satisfying the symmetry condition~\eqref{eq:symmetry} and the Lévy condition~\eqref{eq:conditions.Levy.kernel}.  Among them lies the operator~$\log(I-\Delta)$, whose kernel behaves as~$|x-y|^{-N}$ for~$x\sim y$. It turns out that this singular property is what characterizes strong hypercontractivity within this family.

We could perform our analysis for Lévy-type operators~$\mathcal{L}$ defined by a measure instead of a kernel; we assume here, instead, that the defining measure is absolutely continuous with respect to Lebesgue's measure, and hence that it is given by a kernel, for the sake of simplicity.

In the probabilistic literature, operators of the form~\eqref{general-L} are called Lévy operators only when they are translation invariant. The more general operators we consider, although not translation invariant, share many properties with genuine Lévy operators and have attracted significant attention in recent years because their analysis presents additional mathematical challenges, partly due to the unavailability of Fourier-transform techniques. With a mild abuse of terminology, we refer to them simply as Lévy operators throughout the paper.

\noindent \textbf{Smoothing effects. }
The heat semigroup~$e^{-t(-\Delta)}$ satisfies an~$L^p$--$L^\infty$ smoothing effect for all~$p\geq 1$: there is a constant~$c_{N,p}$ such that
$$
    \|e^{-t(-\Delta)}u_0\|_\infty\le c_{N, p} t^{-\frac{N}{2p}}\|u_0\|_p \quad \text{for all } t>0.
$$
An analogous property holds for the semigroup~$e^{-t(-\Delta)^{\sigma/2}}$ associated to the fractional Laplacian, defined by
$$
    (-\Delta)^{\sigma/2} f(x) = C_{N,\sigma} \text{P.V.} \int_{\RR^N} \frac{f(x) - f(y)}{|x-y|^{N+\sigma}}\, dy,\quad 0<\sigma<2,
$$
which satisfies that for all~$p \geq 1$ there is a constant~$c_{N, p, \sigma}$ such that
$$
    \|e^{-t(-\Delta)^{\sigma/2}} u_0\|_\infty \leq c_{N, p, \sigma} t^{-\frac{N}{\sigma p}}\|u_0\|_p \quad \text{for all } t>0.
$$
Following the standard usage in the literature, we say that the semigroups~$e^{-t(-\Delta)}$ and~$e^{-t(-\Delta)^{\sigma/2}}$ are ultracontractive. With a slight abuse of terminology, we will also refer to their infinitesimal generators,~$-\Delta$ and
$(-\Delta)^{\sigma/2}$, as ultracontractive.

Let us now consider operators~$\mathcal{L}_0$ of the form~\eqref{general-L} that are: of \emph{order zero}, i.e., given by a kernel $J$  integrable in the second variable; and \emph{translation invariant}, i.e., with~$J(x,y)=\widetilde J(x-y)$ for some nonnegative symmetric function~$\widetilde J$. Let us fix~$\|\widetilde{J}\|_1=1$. These operators do not define an ultracontractive semigroup. Indeed, when~$\mathcal{L}=\mathcal{L}_0$, the solution to~\eqref{problem-L} is given by~$u(t) = e^{-t} u_0 + v(t)*u_0$ for some function~$v$, the so-called \emph{regular} part of the heat kernel; see for instance~\cite{ChasseigneChavesRossi2007}. Therefore,~$u(t)$ is only as good as the initial datum is; it does not gain integrability nor regularity.

These are the two expected behaviours for the evolution problem~\eqref{problem-L}: either solutions become instantaneously smooth, as happens for the Laplacian and fractional Laplacians, or they retain exactly the same level of regularity---or irregularity---as the initial data, as in the case of integrable kernels. These two regimes sit at opposite ends of the smoothing spectrum. A natural question then arises: are there operators with intermediate smoothing properties? To investigate this, we seek an example lying between these two extremes within the class of translation-invariant operators, a setting in which the Fourier transform becomes available as a powerful analytical tool.

\noindent \textbf{Translation-invariant operators. }
When~$\mathcal{L}$ is a translation invariant operator, the solution to the Cauchy problem~\eqref{problem-L} is given by the convolution in space of the initial datum with the fundamental solution of the problem,~${H}_t$. This special solution can be defined via the Fourier transform,
\begin{equation}\label{eq:definition.fundamental.solution}
    \widehat{{H}_t}(\xi)=e^{-tm(\xi)},
\end{equation}
where~$m$ is the Fourier symbol of~$\mathcal{L}$. In order to prove the ultracontractivity of the semigroup~$e^{-t\mathcal{L}}$ for such operators, it is enough to check that~$e^{-tm}\in L^1(\mathbb{R}^N)$ for all~$t>0$. It is reasonable to begin by restricting ourselves to radial, continuous symbols~$m$, requiring also~$m(0)=0$ to ensure mass conservation. Therefore, the integrability of~$e^{-tm}$ is determined by the behaviour of~$m(\xi)$ at infinity. In the above examples,~$\mathcal{L}=-\Delta,\,(-\Delta)^{\sigma/2},\,\mathcal{L}_0$, the Fourier symbols are, respectively,
$$
    m(\xi)=|\xi|^2,\, |\xi|^{\sigma},\, 1 -\widehat{J_0}(\xi).
$$
The first two grow polynomially, hence the integrability condition is satisfied. The third one is bounded, and the condition fails. A natural candidate for the intermediate situation we are looking for is a multiplier with logarithmic growth at infinity, for instance~$m(\xi) = \log(1 + |\xi|^2)$.  This particular choice leads to what will be our model case throughout this work, the operator~$\log(I-\Delta)$.

\noindent \textbf{The model case:~$\log(I-\Delta)$. Eventual ultracontractivity and hypercontractivity. }
Since the symbol of this model case is~$m(\xi)=\log(1+|\xi|^2)$, then
\begin{equation}
    \widehat{H_t}(\xi) = e^{-t\log(1 + |\xi|^2)} = (1+|\xi|^2)^{-t}\in L^1(\mathbb{R}^N)\quad \textup{if and only if } t>N/2.
\end{equation}
From this eventual boundedness of~$H_t$ we obtain what is called \emph{eventual ultracontractivity}: solutions become bounded after some finite time. Indeed,  
\begin{equation}
    \|e^{-t\log(I-\Delta)}u_0\|_\infty\le A_p(t) \|u_0\|_p \quad \text{for } p\geq 1,\, t>t_\infty=\frac{N}{2p},
\end{equation}
for some function~$A_p(t) > 0$. Up to time~$t_\infty$ we have~\emph{hypercontractivity}: 
a gradual gain of integrability. Furthermore, we will prove that~$H_t\in L^{\frac{N}{N-2t}, \infty}(\mathbb{R}^N)$ if~$t<N/2$, whence, by weak Young's inequality,
\begin{equation}
    \|e^{-t\log(I-\Delta)}u_0\|_{\frac{Np}{N-2pt}}\le A_p(t)\|u_0\|_p\quad \text{ for } p > 1,\, 0<t<\frac{N}{2p},
\end{equation}
so that any solution with initial data in~$L^p(\RR^N)$,~$p>1$, is at time~$t=t_\infty$ in every~$L^q(\RR^N)$,~$q\in[p,\infty)$. 
We will refer to this property as~\emph{strong hypercontractivity}. Despite its close connection to eventual ultracontractivity, this notion does not necessarily entail the latter for general operators. When the gain of integrability arises instantaneously, it is referred to as~\emph{supercontractivity}. Let us mention that solutions become as regular as desired, not just more integrable, for~$t$ big enough; see~\Cref{rk:higher.regularity}. 

It turns out that~$\log(I-\Delta)$ is a Lévy operator of the form~\eqref{general-L}. For this operator, and any other translation-invariant examples, we already mentioned a simple way to measure the differentiability, and hence the smoothing properties: the behaviour of its Fourier multiplier at infinity. How do we measure it when the Fourier transform is not available? If~$\mathcal{L}$ is a Lévy operator of the form~\eqref{general-L}, the answer comes through the singularity of the kernel.

\noindent \textbf{Singular kernels.} Observe that the infinitesimal generators~$(-\Delta)^{\sigma/2}$,~$\sigma\in(0,2)$, require a certain amount of regularity of a function~$f$ in order for~$(-\Delta)^{\sigma/2}f$ to be well defined. Indeed, the operator~$(-\Delta)^{\sigma/2}$ can be written in the form~\eqref{general-L}, with kernel
\[
J(x,y) = C_{N,\sigma}\,|x-y|^{-(N+\sigma)}, \qquad C_{N,\sigma}>0,
\]
so that~$f$ must at least belong to~$C^{\sigma+\varepsilon}(\mathbb{R}^N)$ for some~$\varepsilon>0$ to ensure that~$(-\Delta)^{\sigma/2}f$ is pointwise meaningful. In contrast, for zero-order translation-invariant operators~$\mathcal{L}_0$, no regularity is required from the functions they act upon for~$\mathcal{L}_0 f$ to be properly defined. The underlying intuition is that the more singular the kernel becomes at short distances, the more regular the solution must be in order to compensate for the singularity. This compensatory mechanism is precisely what leads to the smoothing effects observed in the aforementioned evolution equations.

The operator~$\log(I-\Delta)$ falls into the class of nonlocal operators of type~\eqref{general-L} with a kernel that behaves like~$|x-y|^{-N}$ when~$x$ and~$y$ are close. This places it in an intermediate position between the two previous cases, precisely at the threshold of integrability:~$|x|^{-N-\delta}$ is locally integrable if~$\delta<0$,  but not when~$\delta>0$.
We will show that all operators of the form~\eqref{general-L} whose kernels have the same critical singularity share smoothing properties analogous to those of~$\log(I-\Delta)$. This motivates the terminology \emph{$0^+$-order operators} to denote them; in the literature, they are also referred to as \emph{near-zero order} or \emph{geometrically stable} operators.

An alternative way to assess the differentiability of an operator, more natural within the functional‑analytic framework adopted here, is to quantify it through appropriate Sobolev inequalities.

\noindent\textbf{Sobolev and logarithmic Sobolev inequalities.} We define the Dirichlet form~$\mathcal{E}(\cdot, \cdot)$ associated to an operator of the form~\eqref{general-L} by
\begin{equation}
    \mathcal{E}(f,g) = \frac12\int_{\RR^N} \int_{\RR^N}(f(x) - f(y))(g(x) - g(y))J(x,y)\, dxdy.    
\end{equation}
If the kernel~$J$ is comparable to the one of the fractional Laplacian~$(-\Delta)^{\sigma/2}$,~$\sigma \in (0,2)$, and~$N>\sigma$, Hardy-Littlewood-Sobolev's inequality holds: there is a constant~$c>0$ such that
\begin{equation} \label{eq:sobolevembedding}
    \|f\|_{\frac{2N}{N-\sigma}}^2 \leq c\left(\|f\|_2^2 + \mathcal{E}(f,f)\right) \quad \text{for all }f\in H^{\sigma/2}(\mathbb{R}^N).
\end{equation}
The same inequality is true for the case~$\sigma=2$, corresponding to the operator~$-\Delta$, if~$N>2$, taking as Dirichlet form
\[
    \mathcal{E}(f,g) = \frac12\int_{\RR^N}\nabla f(x)\cdot\nabla g(x)\,dx.    
\]
Observe that the more singular the operator ---its singularity quantified by~$\sigma \in (0,2]$--- the greater the gain in integrability for functions with finite energy. In fact, it is widely known, see for instance~\cite{DaviesBook}, that Sobolev inequalities as~\eqref{eq:sobolevembedding} imply ultracontractivity.

For the case of zero-order operators, there is no gain in integrability. One can easily check that~$\mathcal{E}(f,f) \leq c \|f\|_2^2$ for any~$f \in L^2(\RR^N)$, whence no inequality of the type~\eqref{eq:sobolevembedding} can be available. 

What is then expected of operators like~$\log(I-\Delta)$? They are not ultracontractive, so the Sobolev inequality~\eqref{eq:sobolevembedding} cannot be satisfied. But we expect \emph{some} kind of Sobolev inequality. The answer lies again in an intermediate case: a logarithmic Sobolev inequality
\begin{equation}
 \|f^2\log(f^2)\|_1-\|f\|_2^2\log(\|f\|_2^2)    \leq C\|f\|_2^2 + D \mathcal{E}(f,f).
\end{equation}
However, one of the main insights of this paper is that the above inequality, while sufficient to characterize a certain form of hypercontractivity, as shown by Gross~\cite{GrossLogSobolev}, falls short of describing strong hypercontractivity. To recover such property, \emph{a family} of~$p$--logarithmic Sobolev inequalities of the form 
\begin{equation}
    \||f|^p\log(|f|^p)\|_1-\|f\|_p^p\log(\|f\|_p^p) \leq C(p)\|f\|_p^p + D(p) \mathcal{E}(|f|^{p-2}f,f)
\end{equation}
is required to hold for all $p\ge2$. The asymptotic behaviour of $C(p)$ and $D(p)$ for $p\to \infty$ is what will matter for strong hypercontractivity to appear.

There is no direct mechanism for transferring hypercontractivity from one operator to another. Logarithmic Sobolev inequalities, however, are energy estimates, and operators of the form~\eqref{general-L} with comparable kernels satisfy the same such estimates. Our first aim is to show that general hypercontractivity is equivalent to the existence of an appropriate family of logarithmic Sobolev inequalities of the above type. This reduces the problem: hypercontractivity for any $0^+$-order operator follows once it is proved for a single model in the class. We therefore focus on $\log(I-\Delta)$.

\subsection{Main results}

We start our work by defining the concepts of entropy, hypercontractivity and~$p$–logarithmic Sobolev inequalities. We next prove~\Cref{re:hypertolog}, showing that hypercontractivity for some~$p\geq2$ implies a corresponding~$p$–logarithmic Sobolev inequality. In~\Cref{re:logtohyper} we establish the reciprocal statement, thus proving the equivalence between hypercontractivity and logarithmic Sobolev inequalities. These two statements are, in essence, Gross’s results, suitably reformulated for our framework. Depending on the precise behaviour as~$p\to\infty$ of the coefficients in the family of~$p$–logarithmic Sobolev inequalities, we characterize strong hypercontractivity and provide sufficient conditions for eventual ultracontractivity. Through a duality argument, we extend hypercontractivity to the range~$p\in(1,2)$ in~\Cref{re:dualityresult}. Our approach also leads to supercontractivity and ultracontractivity results. Indeed, as a consequence of our equivalence theorems, we obtain a full characterization of supercontractivity and recover the classical sufficient conditions for ultracontractivity established in the literature. Fully understanding these equivalences between smoothing effects and logarithmic Sobolev inequalities represents the first main contribution of this work.

In order to exploit the previous results, we undertake a detailed analysis of the $0^+$-order operator~$\log(I-\Delta)$, which serves as an ideal model for the theorems above. In~\Cref{re:solutionforlog} we show that solutions to problem~\eqref{problem-L} involving this logarithmic operator gain regularity gradually: as time evolves,~$u(t)$ enters into increasingly more regular Sobolev-type spaces. By Sobolev embeddings, this improvement yields arbitrarily high regularity for sufficiently large times and, in particular, eventual ultracontractivity. A finer, quantitative study of hypercontractivity is carried out in~\Cref{re:loghypercontractivity}, which essentially reduces to a careful examination of the Sobolev embeddings  associated with the Bessel potential $(I-\Delta)^{-t}$ as $t\to 0$. Building on this groundwork, we then apply the equivalence results established earlier to prove in~\Cref{re:logdellog} the logarithmic Sobolev inequalities satisfied by~$\log(I-\Delta)$.

Although the operator~$\log(I-\Delta)$ has been previously considered in the literature, we have not found the associated logarithmic Sobolev inequalities obtained here stated anywhere else. To the best of our knowledge, they constitute the first example of \emph{essential} logarithmic Sobolev inequalities in an infinite measure space, in the sense that our operator does not satisfy any stronger Sobolev inequality of the form~\eqref{eq:sobolevembedding}. Furthermore, the logarithmic Sobolev inequality serves a different purpose here than in the previous literature.  It is typically used to obtain asymptotic results via entropy methods; in our setting, it provides the best energy estimate (Sobolev embedding) expected for $0^+$-order operators. Establishing these inequalities is the second main contribution of the paper.

Then, we examine general operators whose kernels satisfy~$J(x,y) \sim |x-y|^{-N}$ for~$x\sim y$. We prove that such operators enjoy strong hypercontractivity, but fail to be~supercontractive, and hence also ultracontractive. The proof relies essentially on the equivalence established earlier, together with a fine analysis of the model operator $\log(I-\Delta)$. The key idea is to transfer hypercontractivity via logarithmic Sobolev inequalities for operators with comparable kernels. We also show that the result is optimal: regularization is necessarily gradual. This smoothing effect for Lévy operators that need not be translation invariant is the third main contribution of the paper. 

Finally, we turn our attention to translation-invariant operators, whose kernels take the form $J(x,y)=\widetilde J(x-y)$,
with~$\widetilde J$ a symmetric function. We show in \Cref{re:commutationtheorem} that if $\widetilde{J}(z) \geq c|z|^{-N}$ for some $c>0$ and $|z| \ll 1$, then solutions to~\eqref{problem-L} for $\mathcal{L}$ defined by the kernel~$\widetilde{J}$ enter into increasingly more regular Sobolev-type spaces, just like the corresponding solutions for~$\log(I-\Delta)$. Thus, solutions improve first in integrability and then in differentiability after getting bounded. Finally, we show in~\Cref{re:summaryresult} that this gradual smoothing phenomenon only appears if the kernel defining the operator behaves as $|x-y|^{-N}$ for small interactions. Any radial perturbation making the kernel slightly more or less singular breaks this behaviour: if $\ell(|z|) = \widetilde{J}(|z|)|z|^N$ verifies~$\lim_{z \to 0} \ell(z) = \infty$, solutions get bounded and $\cont^\infty$ instantaneously, just as when~$\mathcal{L} = (-\Delta)^{\sigma/2}$ for~$\sigma \in (0,2]$; and if~$\lim_{z \to 0} \ell(z) = 0$, solutions remain only as regular as the initial data, just as when $\mathcal{L}$ is defined via an integrable kernel. These results in the translation-invariant setting constitute the fourth and final main contribution of the paper.

\subsection{Precedents} \label{subse:precedents}

A progressive improvement in the regularity of solutions is highly unusual for parabolic equations and, to the best of our knowledge, essentially new in the context of infinite measures such as the Lebesgue measure. The only example we have found in the literature of hypercontractivity with respect to the Lebesgue measure is precisely that of the operator~$\log(I-\Delta)$, a case treated via subordination theory, as will be discussed below. The logarithmic Sobolev inequalities established here are likewise new for operators defined on infinite measure spaces in situations where no Sobolev embedding is available.

Hypercontractivity was a known and widely studied phenomenon in the context of quantum mechanics. It already appeared for the Ornstein–Uhlenbeck semigroup~$e^{-t\mathcal{M}}$, defined by the Fokker-Plank equation (harmonic oscillator)
\[
    \partial_tu=-\mathcal{M} u:=\Delta u-x\cdot\nabla u.
\] 

Let~$d\gamma(x)=(4\pi)^{-N/2}e^{-|x|^2/4}\,dx$  denote the Gaussian measure. Nelson~\cite{NelsonFreeMarkoff} proved that 
\[
    \begin{gathered}
        \|e^{-t\mathcal{M}}f\|_{L^q(d\gamma)}\le \|f\|_{L^p(d\gamma)} \quad \text{for every } t\ge \frac12\log\left(\frac{q-1}{p-1}\right),\quad 1<p\le q<\infty,\quad\textup{where}\\
        \|f\|_{L^p(d\gamma)}=\left(\int_{\mathbb{R}^N}|f(x)|^p\,d\gamma(x)\right)^{1/p}.
    \end{gathered}
\]
Observe that this hypercontractivity estimate can be written in the form
\[
    \|e^{-t\mathcal{M}}f\|_{L^{q(t)}(d\gamma)}\le \|f\|_{L^p(d\gamma)}, \qquad q(t)=1+e^{2t}(p-1).
\]
Thus, solutions to the Fokker-Plank equation get better integrability with respect to the Gaussian measure as time advances, but they never become bounded.  Gross proved later in~\cite{GrossLogSobolev} that this particular hypercontractivity is equivalent to the Sobolev inequality of logarithmic type
\[
    \|f^2\log(f^2)\|_{L^1(d\gamma)}-\|f\|_{L^2(d\gamma)}^2\log(\|f\|_{L^2(d\gamma)}^2)   \leq C\|f\|^2_{L^2(d\gamma)}  + D\|\mathcal{M}^{1/2}f\|^2_{L^2(d\gamma)},
\]
where the left-hand side is the so-called \emph{entropy} of~$f$. This work was the first to establish the equivalence between hypercontractivity and a logarithmic Sobolev inequality.

That paper already includes what we denote here as~$p$--logarithmic Sobolev inequalities. When an operator~$\mathcal{L}$ satisfies a~$p$--logarithmic Sobolev inequality, Gross refers to it as a \emph{Sobolev generator of index~$p$}. However, in his work, these inequalities are used only as an intermediate technical tool for establishing Nelson's hypercontractivity, and the fundamental 
importance of this entire family of inequalities has until now gone unrecognized; see~\Cref{se:equivalence} for further details.

It is worth stressing that the hypercontractivity exhibited by the Ornstein–Uhlenbeck semigroup is \emph{intrinsic}: the operator is hypercontractive when measured with respect to its ground state, in this case the Gaussian function. A substantial portion of the literature on hypercontractivity, supercontractivity, and ultracontractivity focuses precisely on such intrinsic properties of semigroups. For a comprehensive treatment of intrinsic hypercontractivity, supercontractivity, and ultracontractivity, we refer the reader to~\cite{DaviesSimonIntrinsic}.

The contractivity properties of subordinate semigroups, particularly in contexts involving infinite measures such as Lebesgue's measure, have been extensively analysed in the literature. Such semigroups arise naturally within the framework of Bernstein functional calculus, see for instance the monograph~\cite{Schilling-Song-Vondracek'12}. A systematic treatment of eventual ultracontractivity within this framework was first established in~\cite[Theorem~3.1]{Bendikov-Coulhon-SaloffCoste'07}. In this theory, the operator~$\log(I-\Delta)$ appears naturally as the functional calculus realization of the mapping~$s \mapsto \log(1+s)$ applied to the positive operator~$-\Delta$. This theory also allows to establish functional inequalities of Nash or super-Poincar\'e type, like the ones obtained in~\cite{GentilMaheuxNash}. They are then used to characterize hypercontractivity in~\cite[Proposition~13]{Schilling-Wang'12}. Nevertheless, whilst there is a direct relation between Nash, super-Poincar\'e inequalities and ultracontractivity, it seems that they are not sufficient if one wants to prove eventual ultracontractivity; see~\Cref{se:extensions} for a more detailed explanation.

For the Laplacian with respect to the Lebesgue measure, logarithmic Sobolev inequalities are well understood, and the optimal constant is now known; see, for  instance,~\cite{DelPinoDolbeault1,DelPinoDolbeault2,WeisslerLog}. Several developments have also been made in the nonlocal setting; for instance, logarithmic Sobolev inequalities for the fractional Laplacian have been established in~\cite{ChatzakouRuzhanskyFracLog, VivekSahuFraclog, XiaoSharp}. It is important to note that in all these works the logarithmic Sobolev inequality is not essential, in the sense that a stronger (fractional) Sobolev inequality also holds. We also refer to~\cite{Wang-Wang'13}, where the authors develop the fractional analogue of the results for the Fokker–Planck equation

One of the main sources of inspiration for the present work is~\cite{ArturoCristinaRegularizingEffect}, where it is shown that, for translation-invariant Lévy operators, the threshold for ultracontractivity occurs precisely when the function~$\widetilde J$ defining the kernel satisfies
$\widetilde{J}(z) \sim|z|^{-N}$ near the origin. Further related developments include the analysis in~\cite{ArturoCorreaNearZero} of the energy space associated with $0^+$-order Lévy operators yielding, among other results, a Hardy inequality and an embedding into a Lorentz space; and the proofs in~\cite{jarohs2025continuitysolutionsequationsweakly,KassmannMimicaIntrinsicScaling} that solutions to a related elliptic problem with Dirichlet conditions are bounded and continuous.

Finally, it is worth mentioning that the operator $\log(I-\Delta)$ can be understood as the following limit of fractional relativistic Schrödinger operators $(I-\Delta)^{\alpha}-I$,
\begin{equation}
    \log(I-\Delta) = \lim_{\alpha \to 0^+} \frac{(I-\Delta)^{\alpha} -I}{\alpha}.
\end{equation}
These operators are well known and widely studied; see, for example,~\cite{fall_sharp_2014, fall_unique_2015,frank_hardy-lieb-thirring_2007,herbst_spectral_1977,LiebLossbook,roncal_carleman_2023} for a representative (though not exhaustive) list of references.
Observe that $\log(I-\Delta)$ can also be seen as the derivative at~$\alpha=0$ of the Bessel operators~$(I-\Delta)^{\alpha}$. We do not make use of these two limits to obtain information about~$\log(I-\Delta)$. However, what will be of importance is that $-\log(I-\Delta)$ is the infinitesimal generator associated to the Bessel potential~$(I-\Delta)^{-t}$ semigroup. 

\subsection{Outline of the paper}

In~\Cref{se:preliminaries} we introduce several necessary concepts and recall some useful tools already available in the literature.~\Cref{se:equivalence} is devoted to proving the equivalence between smoothing effects and families of logarithmic Sobolev inequalities. In~\Cref{se:logoperator} we focus on the model operator~$\log(I-\Delta)$, applying the equivalence established above to obtain a new family of logarithmic Sobolev inequalities.~\Cref{se:generaloperators} addresses general operators, making extensive use of the results from the preceding sections. In~\Cref{se:convolution} we establish eventual ultracontractivity under the additional assumption of translation invariance, and we identify the borderline behaviour~$|z|^{-N}$ at the origin as the threshold for this gradual regularization phenomenon. Finally,~\Cref{se:extensions} contains additional comments and possible extensions. For completeness, an~\hyperref[se:existence]{Appendix} summarizes the relevant existence theory.

\section{Preliminaries} \label{se:preliminaries}

\subsection{Lévy operators}
We consider purely nonlocal Lévy operators of the form~\eqref{general-L}, defined by a nonnegative Lévy kernel~$J$ satisfying
\begin{gather}
    \label{eq:symmetry}
    J(x,y) = J(y,x)\quad\text{for all } (x,y) \in \RR^{2N} \setminus \{x = y\}, \\
    \label{eq:conditions.Levy.kernel}
       \sup_{x\in\mathbb{R}^N}\int_{\mathbb{R}^N}J(x,y)\min\{1,|x-y|^2\}\,dy<\infty.
\end{gather}
It has an associated Dirichlet form,
\begin{equation}
    \mathcal{E}(f,g) = \langle \mathcal{L}f, g\rangle = \frac12\int_{\RR^N} \int_{\RR^N}(f(x) - f(y))(g(x) - g(y))J(x,y)\, dxdy,
\end{equation}
which defines the energy~$\overline{\mathcal{E}}(f)=\mathcal{E}(f,f)$, and the Sobolev space
\begin{equation}
    H_{\mathcal{L}}(\RR^N)=\{f\in L^2(\RR^N)\,:\, \overline{\mathcal{E}}(f)<\infty\},
\end{equation}
with norm
\begin{equation}
    \|f\|_{H_{\mathcal{L}}}=\|f\|_2+\overline{\mathcal{E}}^{1/2}(f).
\end{equation}
For the special case~$\mathcal{L}=(-\Delta)^{\sigma/2}$,~$H_\mathcal{L}(\RR^N)$ reduces to the standard fractional Sobolev space~$H^{\sigma/2}(\RR^N)$.

\subsection{Concept of solution}

Recall that our aim is to study the Cauchy problem~\eqref{problem-L}. Let us define the concept of solution used throughout this work, sometimes called \emph{mild solution}.

\begin{Definition} \label{def:L2solutions}
    Given~$u_0\in L^2(\RR^N)$,  we say that~$u$ is a (mild) solution to problem~\eqref{problem-L} if
    \begin{equation}
        u \in \cont([0,\infty); L^2(\RR^N)) \cap \cont^1((0,\infty); L^2(\RR^N)) \cap \cont((0,\infty); H_\mathcal{L}(\RR^N))
    \end{equation}
    and satisfies the abstract ODE
    \begin{equation}
        \dfrac{du}{dt}+\mathcal{L}u =0\quad\text{in }(0,\infty), \qquad u(0) = u_0.
    \end{equation}
\end{Definition}

These solutions are also clearly weak solutions in the standard sense. Indeed, if we multiply the equation by a smooth function~$\phi$, compactly supported in~$\RR^N\times(0,\infty)$, and then integrate by parts, we find that~$u$ satisfies
\begin{equation}
    \int_0^\infty\int_{\mathbb{R}^N}u(x,t)\partial_t\phi(x,t)\,dxdt=\int_0^\infty\mathcal{E}(u,\phi)(t)\,dt.
\end{equation}

For~$\mathcal{L}$ a Lévy operator, if the initial datum belongs to~$L^1(\RR^N)\cap L^2(\RR^N)$,  then~$\|u(t)\|_1 \leq \|u_0\|_1$ and~$u\in\cont([0,\infty); L^1(\RR^N))$. If~$u_0 \in L^p(\RR^N) \cap  L^2(\RR^N)$ for some~$p \in (1, \infty]$, then we also have~$\|u(t)\|_p \leq \|u_0\|_p$. For nonnegative initial data~$u_0\geq 0$, we have that~$u(t) \geq 0$. These properties imply that the solution operator~$e^{-t\mathcal{L}} \colon L^2(\RR^N) \to L^2(\RR^N)$ defined as~${e^{-t\mathcal{L}}u_0 \coloneq u(t)}$ is a \emph{symmetric Markov semigroup}, which is the natural framework in which to tackle the equivalence question of~\Cref{se:equivalence}.

\begin{Definition}
    We say that the solution operator~${e^{-t\mathcal{L}}\colon L^2(\RR^N) \to L^2(\RR^N)}$ is a \emph{symmetric Markov semigroup} if it is a contraction in~$L^\infty(\RR^N)$ and preserves  positivity for all~$t\geq 0$. 
\end{Definition}

For initial data just in~$L^p(\RR^N)$, the notion of weak solution is not the appropriate one, since it is not possible to prove that the energy of a mild solution is bounded. However, as long as the Lévy kernel of the operator is of the form~$J(x,y) = \widetilde{J}(x-y)$, or~$J(x,y)$ is less singular than~$|x-y|^{-(N+1)}$ for~$x\sim y$, the pointwise value~$\mathcal{L}f(x)$ is well defined for~$f$ smooth and bounded. Thus, for~$u_0 \in L^p(\RR^N)$, 
we use an approximation by~$L^2\cap L^p$ mild solutions and the estimate~$\|u(t)\|_p \leq \|u_0\|_p$ to pass to the limit and show that there exists a \emph{very weak solution}; that is, a solution~$u(t)$ with initial data~$u_0$ in~$L^p(\RR^N)$ satisfying
\[
    -\int_0^\infty\int_{\mathbb{R}^N}u\partial_t\varphi+\int_0^\infty\int_{\mathbb{R}^N}u\mathcal{L}\varphi=0\quad\text{for all }\varphi\in C^\infty_{\textup{c}}(\mathbb{R}^N\times(0,\infty)).
\]

In the case of translation-invariant operators, the convolution~$u(t) = {H}_t*u_0$, with~${H}_t$ as in~\eqref{eq:definition.fundamental.solution} is well defined for any initial data~$u_0 \in L^p(\RR^N)$,~$p \in [1, \infty]$. 
When~$p=2$, this convolution coincides with the mild solution defined above for the operators considered here. The only operator mentioned in this paper where both notions of solution may not coincide for~$p=2$ is~$\mathcal{L} = \log(-\Delta)$. We explain this with more care at the end of~\Cref{se:logoperator}.

There is no loss of generality in considering only nonnegative solutions. For the general case we make the decomposition~$u_0 = (u_0)_+ - (u_0)_-$ and consider the nonnegative solutions~$u^+(t)$ and~$u^-(t)$ associated respectively to the initial data~$(u_0)_+$ and~$(u_0)_-$ (remember that if~$u_0 \geq 0$, then~$u(t) \geq 0$ for all~$t\geq 0$). Then,  due to linearity,~$u(t) = u^+(t) - u^-(t)$, so we can deduce the behaviour of~$u(t)$ from that of~$u^+(t)$ and~$u^-(t)$. For more information about the existence of solutions and their basic properties see the~\hyperref[se:existence]{Appendix}.

\subsection{Weak Lebesgue spaces and weak Young's inequality}
We define the \emph{weak Lebesgue space}~$L^{p,\infty}(\RR^N)$ (also known as the \emph{Marcinkiewicz space}) as the set of measurable functions~$f$ defined in $\RR^N$ such that
\begin{equation}
    \|f\|_{p,\infty} = \sup\left(\lambda \mu\{x \in \RR^N: |f(x)|>\lambda\}^{1/p} \right) < \infty, 
\end{equation}
where $\mu(A)$ denotes the measure of the set $A$. The functional $\|\cdot\|_{p,\infty}$ defines a quasinorm in this space. Weak Lebesgue spaces are heavily used in this work combined with \emph{weak Young's inequality}, which we state below.  

The usual \emph{Young's inequality} states that, for $p,q,r \geq 1$ satisfying 
\begin{equation} \label{eq:youngcondition}
    1+ \frac 1r = \frac 1p + \frac 1q,
\end{equation}
and $f \in L^p(\RR^N)$, $g \in L^q(\RR^N)$, the convolution $f*g$ verifies
\begin{equation}
    \|f*g\|_r \leq \|f\|_p \|g\|_q.
\end{equation}
If $p,q,r > 1$, this inequality can be extended to \emph{weak Young's inequality}
\begin{equation}
    \|f*g\|_r \leq c \|f\|_{p,\infty} \|g\|_q
\end{equation}
valid for all $f \in L^{p,\infty}(\RR^N)$, $g \in L^q(\RR^N)$, where $c>0$ is a constant depending on the exponents and the dimension. A natural question arises in this context: is weak Young's inequality optimal? That is, assume that for any $r>q>1$ there exists some $C>0$ such that
\begin{equation}
    \|f*g\|_r \leq C \|g\|_q
\end{equation}
for any~$g \in L^q(\RR^N)$. Does this imply that~$f$ belongs to $L^{p,\infty}(\RR^N)$ for the corresponding $p$ satisfying~\eqref{eq:youngcondition}? The result trivially holds under the assumption that $r=\infty$, see~\Cref{re:translationequivalence2}. Then, we will prove in~\Cref{re:optimalweak} that this statement also holds if $f$ is a nonnegative radial function, nonincreasing in $|x|$. We believe that this result is likely available in the literature, although we have been unable to find a reference. We consider this lemma to be of independent interest.

\subsection{Bessel potentials spaces}
Let us consider the Sobolev space~$W^{k,p}(\RR^N)$, defined as the set of functions in~$L^p(\RR^N)$ with weak derivatives up to order~$k \in \mathbb{N}$ also in~$L^p(\RR^N)$. There are several ways to extend this definition to a fractional framework. A natural possibility is to interpolate between the integer~$W^{k,p}(\RR^N)$ spaces: real interpolation yields the fractional Sobolev spaces~$W^{\sigma,p}(\mathbb{R}^N)$; interpolating through the complex method gives the so-called \emph{Bessel potential spaces}, denoted by~$L^p_\sigma(\RR^N)$. These two interpolations coincide in the Hilbert case, namely~$L^2_\sigma(\RR^N) = W^{\sigma,2}(\RR^N)$, which is commonly denoted by~$H^\sigma(\RR^N)$. We will focus on the Bessel potential spaces. The name comes from an equivalent definition that we state below.
\begin{Definition}
    Let~$p>1$ and~$\sigma \geq 0$. We define the \emph{Bessel potential spaces}~$L^p_\sigma(\RR^N)$ as
    \begin{equation}
        L^p_\sigma(\RR^N) = \{f \in L^p(\RR^N): (I-\Delta)^{\sigma/2} f \in L^p(\RR^N) \},
    \end{equation}
    where~$(I-\Delta)^{\sigma/2}$ is the \emph{Bessel operator}, defined as
    \begin{equation}
        \big[(I-\Delta)^{\sigma/2}f\big]\,\widehat{\;}\,(\xi) =  (1+|\xi|^2)^{\sigma/2} \widehat f.
    \end{equation}
    The space~$L^p_0(\RR^N)$ is understood as~$L^p(\RR^N)$.
\end{Definition}
Although these spaces can also be defined for~$p=1$, this case is usually left out in the literature since the subsequent theory can be difficult to deal with. In these instances the space~$L^1(\RR^N)$ is often replaced by the corresponding real Hardy space. 

Let us now define the Bessel potentials, inverse of the Bessel operators.

\begin{Definition}
    Let~$p>1$ and~$\sigma > 0$. We define the \emph{Bessel potential}~$(I-\Delta)^{-\sigma/2}$ as
    \begin{equation}
        \big[(I-\Delta)^{-\sigma/2}f\big]\,\widehat{\;}\,(\xi) =  (1+|\xi|^2)^{-\sigma/2} \widehat f.
    \end{equation}
\end{Definition}

Observe that if~$f\in L^p(\RR^N)$, then~$(I-\Delta)^{-\sigma/2}f\in L^p_\sigma(\RR^N)$. 

We also need to define fractional Laplacians and their inverses, the Riesz potentials.

\begin{Definition}
    Let~$\sigma > 0$. We define the \emph{fractional Laplacian}~$(-\Delta)^{\sigma/2}$ as
    \begin{equation}
        \big[(-\Delta)^{-\sigma/2}f\big]\,\widehat{\;}\,(\xi) = |\xi|^{\sigma} \widehat f
    \end{equation}
    for~$f$ a good enough function. We define the \emph{Riesz potential} as its inverse, that is,
    \begin{equation}
        \big[(-\Delta)^{-\sigma/2}f\big]\,\widehat{\;}\,(\xi) = |\xi|^{-\sigma}  \widehat f.
    \end{equation}
\end{Definition}

For more information about the Riesz and Bessel potentials, see for  instance~\cite{BesselII,Adamsbook,AdamsFourierbook,BesselI,GrafakosModern,Steinbook}.

\section{Hypercontractivity and logarithmic Sobolev inequalities} \label{se:equivalence}

Throughout this section, all integrals and~$L^p$ norms could be understood as defined on~$(X,d\mu)$, where~$d\mu$ is a Borel measure on a locally compact, second countable Hausdorff space~$X$. However, for simplicity, we take~$X = \mathbb{R}^N$ and~$\mu = dx$, the Lebesgue measure. A very natural framework for this section is the one of symmetric Markov operators~$e^{-t\mathcal{L}}$, which is the case for all semigroups having as infinitesimal generators operators~$\mathcal{L}$ like the ones considered here.

Let us rigorously define the different versions of \emph{hypercontractivity} and \emph{ultracontractivity} we will be dealing with. We may talk indistinctly about an operator~$\mathcal{L}$ or the associated semigroup~$e^{-t\mathcal{L}}$.

\begin{Definition} \label{contractivity2}
    Let~$\mathcal{L}$ be an operator and let~$u(t)=u(\cdot,t)$ be the solution to problem~\eqref{problem-L}. Then~$\mathcal{L}$ (the  semigroup~$e^{-t\mathcal{L}}$) is:
    \begin{itemize}
        \item \emph{$p$--hypercontractive} if there exist functions~$q:[0, t_\infty)\to[p,\infty)$ and $A_p:[0,t_\infty) \to [0,\infty)$ for some~$t_\infty \in (0, \infty]$ and $p\in[1,\infty)$, such that 
            \begin{equation}\label{def:hypercontractive}
                \|u(t)\|_{q(t)}\le A_p(t)\|u_0\|_p \quad\text{for } 0 \leq t<t_\infty,
            \end{equation}
        with $q$ nondecreasing and satisfying~$q(0)=p$,~$\lim\limits_{t \to t_\infty} q(t) = \infty$, and $A_p(0) \geq 1$.
        \item \emph{Strongly~$p$--hypercontractive} if it is~$p$--hypercontractive with~$t_\infty<\infty$.
        \item\emph{Supercontractive} if there exists a function~$B(t,q) \geq 0$ such that 
        \begin{equation} \label{def:supercontractive}
            \|u(t)\|_{q}\le B(t,q)\|u_0\|_2 \quad\text{for every } 2 \leq q < \infty \text{ and } t>0.
        \end{equation}
        \item \emph{Ultracontractive} if it is supercontractive and~$B(t) \coloneq \sup_q B(t,q) < \infty$ for all~$t>0$ so that in particular
        \begin{equation} \label{def:ultracontractive}
                \|u(t)\|_{\infty}\le B(t)\|u_0\|_2 \quad\text{for } t>0.
        \end{equation}
        \item \emph{Eventually ultracontractive} if it satisfies~\eqref{def:ultracontractive} for all~$t\geq t_0$ for some~$0<t_0<\infty$.
    \end{itemize}
\end{Definition}

\begin{Remark}
    There is a reason to talk about~$p$--hypercontractivity and not about~$p$--supercontractivity or~$p$--ultracontractivity. In the usual instances found in the literature, two~$p$--hypercontractivities with two different values of~$p$ imply each other quantitatively, since they are both equivalent to the standard~$2$--logarithmic Sobolev inequality. However, to show strong hypercontractivity we need a family of \emph{independent}~$p$--logarithmic Sobolev inequalities; therefore, the~$p$--hypercontractivities are also expected to be quantitatively independent; see the discussion below. The case of supercontractivity and ultracontractivity is different, since through an interpolation and iteration argument in~\eqref{def:supercontractive} and~\eqref{def:ultracontractive}, standard in the literature, one could substitute the~$2$--norm by any~$p$--norm.
\end{Remark}
\begin{Remark}
    The concept of \emph{eventual supercontractivity} can be defined in a completely analogous way to eventual ultracontractivity. At first glance, one might think that strong hypercontractivity and eventual supercontractivity coincide. However, there is an important distinction: the gradual, quantitative nature of strong hypercontractivity, which may be absent in eventual supercontractivity.
\end{Remark}

In order to define the logarithmic Sobolev inequalities properly, let us first define the~\emph{entropy} of a function.

\begin{Definition}
    Given a measurable function~$f$, we define its \emph{entropy} as
    \begin{equation} \label{eq:entropy}
        \ent(f)=\int_{\RR^N} \Phi(|f|) - \Phi\left(\int_{\RR^N}|f|\right), \quad \text{where }\Phi(s)=s\log s.
    \end{equation}
\end{Definition}
Observe that Jensen's inequality is not available, and thus we cannot ensure nonnegativity of the entropy; however, this is of no importance for this work. We recall next the standard logarithmic Sobolev inequality that gives a control of the entropy of~$f^2$ in terms of the energy of~$f$.

\begin{Definition}
    We say that the operator~$\mathcal{L}$ satisfies a standard \emph{logarithmic Sobolev inequality} if there exist constants~$C \in \RR$ and~$D>0$ such that
    \begin{equation}\label{eq:log-sob-ineq}
        \ent(f^2)\leq C\|f\|_2^2 + D \overline{\mathcal{E}}(f)\quad\text{for every }f\in H_{\mathcal{L}}(\RR^N).
    \end{equation}
\end{Definition}

The equivalence between some smoothing effects and logarithmic Sobolev inequalities has been known for some time. For example, assume that the inequality
\begin{equation}
    \ent(f^2)\leq 2\beta(\varepsilon)\|f\|_2^2 + 2\varepsilon \overline{\mathcal{E}}(f)
\end{equation}
holds for all~$\varepsilon>0$, and~$\beta$ a monotonically decreasing continuous function such that for all~$t>0$,
\begin{equation}
    M(t) = \frac 1t \int_0^t \beta(\varepsilon) \, d\varepsilon < \infty.
\end{equation}
Then the operator~$\mathcal{L}$  is ultracontractive and solutions to~\eqref{problem-L} satisfy
\begin{equation}
    \|u(t)\|_\infty \leq e^{M(t)} \|u_0\|_2 \quad \text{for all } t>0.
\end{equation}
The reciprocal implication is also true; see~\cite{DaviesBook} for more information. It is also known that satisfying the logarithmic Sobolev inequality~\eqref{eq:log-sob-ineq} is equivalent to the following quantitative hypercontractivity property,
\begin{equation}\label{eq:hypercontractivity.Davies}
    \begin{gathered}
        \|u(t)\|_{q(t)} \leq e^{M_p(t)} \|u_0\|_p,\quad\text{with}\\
        q(t) = (p-1) e^{t/D} + 1, \qquad M_p(t) = C \left(\frac 1p + \frac 1{q(t)} \right);
    \end{gathered}
\end{equation}
see~\cite{BarkyBook,GrossLogSobolev}.

The function~$q$ in~\eqref{eq:hypercontractivity.Davies} is monotonically increasing and unbounded. Notice that, no matter the choice of~$C$ and~$D$, we cannot force~$q$ to blow up in finite time. This is a limitation of the technique, hidden in the fact that we are using \emph{only one} logarithmic Sobolev inequality. To overcome it, we are going to require a family of~$p$--logarithmic Sobolev inequalities. 

\begin{Definition}
    Given~$p\ge2$, we say that~$\mathcal{L}$ satisfies a \emph{$p$--logarithmic Sobolev inequality} if  there exist constants~$C(p), D(p)$, with~$D(p)>0$, such that
    \begin{equation}\label{eq:entropy-ineq}
        \ent(|f|^p) \leq C(p) \|f\|_p^p + D(p) \mathcal{E}(|f|^{p-2}f,f)\quad\textup{for every }|f|^{p/2}\in H_{\mathcal{L}}(\RR^N).
    \end{equation}
\end{Definition}

Observe that the~$2$--logarithmic Sobolev inequality reduces to the standard logarithmic Sobolev inequality. The fact that a family of~$p$--logarithmic Sobolev inequalities is useful is already present in the proofs of ultracontractivity and hypercontractivity performed respectively in~\cite{DaviesBook} and~\cite{GrossLogSobolev}, although these families only appear as an intermediate step in the proofs. We claim that having \emph{a family} of independent logarithmic Sobolev inequalities, and not just one, is essential when studying hypercontractivity and ultracontractivity issues, since this allows for much richer behaviours.

Recall Stroock-Varopoulos' inequality,
\begin{equation}
    \overline{\mathcal{E}}(|f|^{p/2}) \leq \frac{p^2}{4(p-1)}  \mathcal{E}(|f|^{p-2}f,f);
\end{equation}
see~\cite{Stroock-1984,Varopoulos-1985}. Combining this with the~$2$--logarithmic Sobolev inequality~\eqref{eq:log-sob-ineq} with~$f$ replaced by~$|f|^{p/2}$, we get
\begin{equation} \label{eq:badlogfamily}
    \ent(|f|^p) \leq C(2) \|f\|_p^p + \frac{D(2)p^2}{4(p-1)} \mathcal{E}(|f|^{p-2}f,f).
\end{equation}
One could use this family of~$p$--logarithmic inequalities and the method explained below to obtain the hypercontractivity result~\eqref{eq:hypercontractivity.Davies}. This is a very particular case in which all the~$p$--logarithmic Sobolev inequalities span from the same one. But in doing this, we are limiting ourselves to a certain family of~$p$--logarithmic Sobolev inequalities, and hence to a particular type of hypercontractivity.

\begin{Remark}
    All~$p$--logarithmic Sobolev inequalities of the form
    \begin{equation}
        \ent(|f|^p) \leq C \|f\|_p^p + \frac{D p^2}{4(p-1)} \mathcal{E}(|f|^{p-2}f,f)
    \end{equation}
    with~$C$ and~$D$ independent of~$p$ are equivalent for \emph{diffusion operators}, that is, operators satisfying
    \begin{equation}
        \mathcal{L}(\Psi(f))=\Psi'(f)\mathcal{L}(f)+\Psi''(f)\Gamma(f),
    \end{equation}
    where~$\Gamma$ is the \emph{carré du champ} (essentially the density) of the Dirichlet form~$\overline{\mathcal{E}}$; see~\cite{BarkyBook}. Nonlocal operators lack this property, which is a major differentiating factor between them and local operators.
\end{Remark}

Let now~$u$ be a solution to~\eqref{problem-L}. Recall that we may assume that~$u\ge0$, and we will do so in the sequel. Also, all the integrals are understood to be over~$\RR^N$. Differentiating~$\displaystyle\int u^{q(t)}(t)$, we formally obtain
\begin{equation} \label{eq:logdiff}
    \frac{d}{dt} \int u^{q(t)}(t)=q'(t)\int u^{q(t)}(t)\log(u(t))+q(t)\int u^{q(t)-1}(t)\partial_t u(t).
\end{equation}
Using that~$u$ is a solution and the definition of entropy~\eqref{eq:entropy}, we can rewrite this identity as
\begin{equation} \label{eq:diff}
    \frac{d}{dt} \int u^{q(t)}(t)  + q(t) \int u^{q(t)-1} \mathcal{L}u(t)  = \frac{q'(t)}{q(t)} \left\{ \text{Ent}(u^{q(t)}(t)) + \left(\int u^{q(t)}(t)\right)  \log \left( \int u^{q(t)}(t)  \right) \right\},
\end{equation}
which is the starting point for the results that follow. To justify these computations, we need the next auxiliary results. The first two can be found in~\cite{DaviesBook}. We add here the proofs for completeness.

\begin{Lemma} \label{re:teclemma}
    Let~$q:[0, t_\infty)\to[2, \infty)$, for some~$t_\infty\in(0,\infty]$, be a~$\cont^1$ function, and let~$u\ge0$ be the solution of~\eqref{problem-L} with~$u_0 \in L^1(\RR^N) \cap L^\infty(\RR^N)$. Then~$\partial_t u(t) u(t)^{q(t)-1}$ and~$u(t)^{q(t)} \log(u(t))$ belong to~$\cont([0, \infty); L^1(\RR^N))$.
\end{Lemma}

\begin{proof}
    By hypothesis,~$\partial_t u(t)$ is continuous with values in~$L^2$, so we need to check that~$u^{q(t)-1}(t)$ is also continuous with values in~$L^2$. Since~$\|u(s)\|_\infty\le \|u_0\|_\infty$ for all~$s\ge0$,
    \begin{equation}
        \begin{aligned}
            |u^{q(s_1)-1}(s_1) - u^{q(s_2)-1}(s_2)| &\leq |u^{q(s_1)-1}(s_1) - u^{q(s_1)-1}(s_2)|+ |u^{q(s_1)-1}(s_2) - u^{q(s_2)-1}(s_2)| \\
            &\leq c_1 |u(s_1) - u(s_2)| + c_2 |q(s_1) - q(s_2)| \, |u(s_2)|^{1/2},
        \end{aligned}
    \end{equation}
    whence the desired continuity in~$L^2(\mathbb{R}^N)$, since
    \begin{equation}
        \begin{aligned}
                \|u(s_1)^{q(s_1)-1} - u(s_2)^{q(s_2)-1}\|_2 &\leq c_1 \|u(s_1) - u(s_2)\|_2 +  c_2 |q(s_1) - q(s_2)| \, \|u(s_2)\|^{1/2}_1\\
                &\leq c_1 \|u(s_1) - u(s_2)\|_2 +  c_2 |q(s_1) - q(s_2)| \, \|u_0\|^{1/2}_1.
        \end{aligned}
    \end{equation}
    To show that~$u^{q(t)}(t) \log(u(t))$ is continuous in~$L^1$ we proceed similarly,
    \begin{equation}
        \begin{aligned}
            |u(s_1)^{q(s_1)}\log(u(s_1)) - u(s_2)^{q(s_2)}\log(u(s_2))|&\leq |u(s_1)^{q(s_1)}\log(u(s_1)) - u(s_2)^{q(s_1)}\log(u(s_2))| \\
            &\quad+ |u(s_2)^{q(s_1)}\log(u(s_2)) - u(s_2)^{q(s_2)}\log(u(s_2))| \\
            &\leq c_1 |u(s_1) - u(s_2)| + c_2 |q(s_1) - q(s_2)| \, |u(s_2)|,
        \end{aligned}
    \end{equation}
    whence
    \begin{equation}
        \begin{aligned}
            \|u(s_1)^{q(s_1)}\log(u(s_1)) - u(s_2)^{q(s_2)}\log(u(s_2))\|_1 &\leq c_1 \|u(s_1) - u(s_2)\|_1 + c_2 |q(s_1) - q(s_2)| \, \|u(s_2)\|_1 \\
            &\leq c_1 \|u(s_1) - u(s_2)\|_1 + c_2 |q(s_1) - q(s_2)| \, \|u_0\|_1,
        \end{aligned}
    \end{equation}
    where we used the~$L^1$-continuity of the solution.
\end{proof}
\begin{Corollary}
    Let~$u\ge0$ be a solution to problem~\eqref{problem-L}. Then,~\eqref{eq:logdiff} is satisfied.
\end{Corollary}

\begin{proof}
    If~$u$ is a smooth function, then
    \begin{equation}\label{eq:differential.version}
        \frac{d}{dt} u^{q(t)}(t)  = q(t)\partial_t u(t) u^{q(t)-1}(t)  + q'(t) u^{q(t)}(t) \log(u(t)).
    \end{equation}
    Integrating in~$(0,t)$ we get
    \begin{equation}\label{eq:integral.version}
        u^{q(t)}(t)  = u_0^p+ \int_0^t \left(q(s)u^{q(s)-1}(s)\partial_t u(s)+ q'(s) u^{q(s)}(s) \log(u(s))\right)\,ds.
    \end{equation}
    An approximation argument using continuous piecewise linear functions of $s$ and~\Cref{re:teclemma} shows that~\eqref{eq:integral.version} is also true in~$L^1$. Hence,~\eqref{eq:differential.version} holds in~$L^1$.
\end{proof}
For the next result we need the following energy lemma for the composition, whose proof follows directly by using the Mean Value Theorem.
\begin{Lemma} \label{re:betterfunction}
    Let~$f \in L^\infty(\RR^N) \cap H_\mathcal{L}(\RR^N)$ and~$\varphi\in \cont^1(\RR)$ such that~$\varphi(0) = 0$. Then~$\varphi(f) \in H_\mathcal{L}(\RR^N)$. In particular,~$|f|^q\in H_\mathcal{L}(\RR^N)$ for any~$q\geq1$.
\end{Lemma}

We are now ready to prove the first main result of this paper.

\begin{Theorem}\label{re:hypertolog}
    Let~$\mathcal{L}$ be~$p$--hypercontractive for some~$p\geq 2$, satisfying~\eqref{def:hypercontractive} with~$q$ an increasing~$\cont^1$ function defined on~$[0, t_\infty)$,~$t_\infty \in (0, \infty]$, with~$q'(0) > 0$, and
    \begin{equation}
        A_p(0)=\lim_{h\to0^+}A_p(h)=1,\qquad A'_p(0)=\lim_{h\to0^+}\frac{A_p(h)-1}{h}<\infty.
    \end{equation}
    Then the~$p$--logarithmic Sobolev inequality~\eqref{eq:entropy-ineq} holds for any~$f \in L^1(\RR^N) \cap L^p(\RR^N) \cap H_\mathcal{L}(\RR^N)$ with
    \begin{equation}
        C(p)=\frac{p^2A_p'(0)}{q'(0)},\qquad D(p)=\frac{p^2}{q'(0)}.
    \end{equation}
\end{Theorem}

\begin{proof}
    Assume first that~$f\ge0$ is bounded, and let~$u$ be the solution with initial datum~$f$.  
    Defining~$U(t) = \displaystyle\int u^{q(t)}(t)$, identity~\eqref{eq:diff} is written as
    \begin{equation}\label{eq:equality.U}
        U'(t)+q(t) \mathcal{E}\big(u^{q(t)-1}(t),u(t)\big)=\frac{q'(t)}{q(t)}\left\{\text{Ent}(u^{q(t)}(t)) + U(t)\log \left(U(t)\right) \right\}.
    \end{equation}
    Rewriting hypercontractivity with this notation, we get
    \begin{equation}
        U(t)^{\frac{1}{q(t)}} \leq A_p(t) U(0)^{\frac{1}{p}}\quad\textup{for all }t\ge0,
    \end{equation}
    with~$q(0) = p$ and~$A_p(0) = 1$. Since both sides coincide at~$t=0$, the derivatives at~$t=0$ verify
    \begin{equation}
        U'(0) \leq \left.\frac{d}{dt} \left( A_p^{q(t)}(t)U(0)^{\frac{q(t)}{p}} \right) \right|_{t=0} = U(0)\left[ \frac{ q'(0)}{p} \log(U(0)) + p A_p'(0)\right].
    \end{equation}
    If~$A_p'(0) = -\infty$, we simply replace~$A_p(t)$ by any bound above satisfying the required hypotheses. Inserting this into~\eqref{eq:equality.U} we obtain
    \begin{equation}
        \text{Ent}(f^{p}) \leq \frac{p^2}{q'(0)} A_p'(0) \int_{\RR^N} f^p + \frac{p^2}{q'(0)} \mathcal{E}(f^{p-1},f).
    \end{equation}
    For~$f\in L^1(\RR^N) \cap L^p(\RR^N) \cap H_\mathcal{L}(\mathbb{R}^N)$, take a sequence~$f_n \in L^1(\RR^N) \cap L^\infty(\RR^N) \cap H_\mathcal{L}(\mathbb{R}^N)$ approximating it. We have that~$f_n$ satisfies the logarithmic Sobolev inequality
    \begin{equation}
        \int_{\RR^N} f_n^p \log{f_n^p} \leq C(p) \|f_n\|_p^p + \|f_n\|_p^p \log{\|f_n\|_p^p} + D(p) \mathcal{E}(f_n^{p-1}, f_n).
    \end{equation}
    In the right-hand side we can trivially take limits, so we only need to examine the limit of the left-hand side. Observe that
    \begin{equation}
        f_n(x) + f_n^p(x) \log{f_n^p(x)} \geq 0.
    \end{equation}
    Then, thanks to Fatou's lemma,
    \begin{equation}
        \int_{\RR^N}(f + f^p \log{f^p}) \leq \liminf_{n \to \infty} \int_{\RR^N} (f_n + f_n^p \log{f_n^p}).
    \end{equation}
    Since
    \begin{equation}
        \lim_{n \to \infty}\int_{\RR^N} f_n = \int_{\RR^N} f,
    \end{equation}
    we get
    \begin{equation}
        \int_{\RR^N} f^p \log{f^p} \leq \liminf_{n \to \infty} \int_{\RR^N} f_n^p \log{f_n^p},
    \end{equation}
    therefore showing that~$f$ also satisfies the~$p$--logarithmic Sobolev inequality.
\end{proof}

Now, let us prove a reciprocal result, which requires stronger hypotheses. While  we required~$p$--hypercontractivity to hold in order to prove the~$p$--logarithmic Sobolev inequality, we need all the~$r$--logarithmic Sobolev inequalities for~$r\geq p$ to prove the~$p$--hypercontractivity.

\begin{Theorem} \label{re:logtohyper}
    Fix~$p \geq 2$. Assume that the~$r$--logarithmic Sobolev inequalities~\eqref{eq:entropy-ineq} hold for all~$r\geq p$. Then~$\mathcal{L}$ is $p$--hypercontractive with the family of exponents satisfying
    \begin{equation}
        q'(t) = \dfrac{q^2(t)}{D(q(t))},\qquad q(0) = p,
    \end{equation}
    and coefficients
    \begin{equation}
        A_p(t)=\exp\left(\int_0^t\frac{C(q(s))}{D(q(s))}\,ds\right)=\exp\left(\int_p^{q(t)}\frac{C(\alpha)}{\alpha^2}\,d\alpha\right).
    \end{equation}
\end{Theorem}

\begin{proof}
    Let~$u_0 \in  L^1(\RR^N) \cap L^\infty(\RR^N)\cap H_\mathcal{L}(\mathbb{R}^N)$, $u_0\ge0$, and let~$u(t)\in L^1(\RR^N) \cap L^\infty(\RR^N)\cap H_\mathcal{L}(\mathbb{R}^N)$ be the solution of problem~\eqref{problem-L} with initial data~$u_0$. Let~$s:[0, \infty)\to[p, \infty)$ be a~$\cont^1$ function to be determined later. Then,
    \begin{equation}\label{eq:inequality.entropy}
        \text{Ent}(u(t)^{s(t)}) \leq C(s(t)) \int_{\RR^N} u(t)^{s(t)} + D(s(t)) \mathcal{E} (u(t)^{s(t)-1}, u(t)).
    \end{equation}
    Writing~$U(t) = \displaystyle\int u^{s(t)}(t)$, we obtain
    \begin{equation}\label{eq:equality.U2}
        U'(t) + s(t) \mathcal{E}\big(u^{s(t)-1}(t), u(t)\big)  \leq \frac{s'(t)}{s(t)} \left\{ C(s(t)) U(t) + D(s(t)) \mathcal{E}\big(u^{s(t)-1}(t), u(t)\big) + U(t) \log \left( U(t) \right) \right\}.
    \end{equation}
    Now choose~$s(t)=q(t)$, where~$q(t)$ is the solution to
    \begin{equation}
        q'(t) = \dfrac{q^2(t)}{D(q(t))},\qquad q(0)=p.
    \end{equation}
    Combining~\eqref{eq:inequality.entropy} and~\eqref{eq:equality.U2}, the terms involving the operator cancel out and we get the differential inequality
    \begin{equation}
        U'(t) \leq \frac{q'(t)}{q(t)} \left\{ C(q(t)) U(t) + U(t) \log \left( U(t) \right) \right\}.
    \end{equation}
    This implies the desired~$p$--hypercontractivity,
    \begin{equation}
        U(t)^{1/q(t)} \leq \exp \left( \int_0^t \frac{C(q(s))}{D(q(s))} \, ds \right) U_0^{1/p}.
    \end{equation}
    Finally, the change of variables~$q(s)=\alpha$ yields
    \begin{equation}
        \int_0^t \frac{C(q(s))}{D(q(s))} \, ds = \int_p^{q(t)} \frac{C(\alpha)}{\alpha^2} \, d\alpha.
    \end{equation}
    We have shown that~$p$--hypercontractivity holds for bounded initial data. If the initial data is just in~$L^p$, the argument works via approximation for the limit solution.
\end{proof}

Notice that, given~$p \geq 2$, satisfying the~$r$--logarithmic Sobolev inequalities for every~$r \geq p$ is equivalent to being~$r$--hypercontractive for every~$r \geq p$. We would like to extend the result to~$p\in(1,2)$. This will be done through a duality argument, which requires showing that the semigroup~$e^{-t\mathcal{L}}$ is self-adjoint, an immediate consequence of~$\mathcal{L}$ being self-adjoint.
 
\begin{Lemma}
    The semigroup operator~$e^{-t\mathcal{L}}:L^2(\RR^N) \to L^2(\RR^N)$ associated to a self-adjoint operator~$\mathcal{L}$ is also a self-adjoint operator.
\end{Lemma}

\begin{proof}
    Let~$u$ and~$v$ be solutions to problem~\eqref{problem-L} with initial data~$u_0, v_0 \in L^2(\RR^N)$ respectively. We wish to prove that
    \begin{equation}
        \int_{\RR^N} u_0 e^{-t\mathcal{L}}v_0=\int_{\RR^N} u_0 v(t) = \int_{\RR^N} u(t) v_0=\int_{\RR^N} v_0 e^{-t\mathcal{L}}u_0\quad\textup{for any }t>0.
    \end{equation}
    Fix~$t>0$, and define
    \begin{equation}
        F(\tau) = \int_{\RR^N} u(\tau) v(t-\tau) \quad \text{ for } 0\leq \tau \leq t.
    \end{equation}
    Then,
    \begin{equation}
        \begin{aligned}
            F'(\tau) &= \int_{\RR^N} u'(\tau) v(t-\tau) - \int_{\RR^N} u(\tau) v'(t-\tau) \\
            &= -\int_{\RR^N} \mathcal{L}u(\tau) v(t-\tau) + \int_{\RR^N} u(\tau) \mathcal{L}v(t-\tau) = 0,
        \end{aligned}
    \end{equation}
    due to the self-adjointness of~$\mathcal{L}$. Thus,~$F$ is constant, and evaluating at~$\tau = 0$ and~$\tau = t$ yields the result.
\end{proof}

Using the self-adjointness of the solution semigroup~$e^{-t\mathcal{L}}$ we can extend the range of hypercontractivity by means of a standard duality argument.

\begin{Proposition} \label{re:dualityresult}
    Let~$u(t) = e^{-t\mathcal{L}} u_0$ be the solution of problem~\eqref{problem-L}. Assume that~$\mathcal{L}$ is~$p$--hypercontractive for every~$p\geq 2$, and let us define~$t_p^q$, with~$q\geq p$, as the time necessary to get from~$L^p$ to~$L^q$, so that
    \begin{equation}
        \|e^{-t_p^q\mathcal{L}} u_0\|_q \leq A_p(t_p^q) \|u_0\|_p.
    \end{equation}
    Then, if~$p' = \frac{p}{p-1}$ denotes the Hölder conjugate of~$p$,
    \begin{equation}
        \|e^{-t_p^q\mathcal{L}} v_0\|_{p'} \leq A_p(t_p^q) \|v_0\|_{q'}.
    \end{equation}
    Hence, we have that~$t_{q'}^{p'} = t_p^q$, and we can define~$A_{q'}(t) = A_p(t)$. Therefore,~$\mathcal{L}$ is~$p$-hypercontractive for every~$1<p\leq2$.

    Furthermore, if~$\mathcal{L}$ is eventually ultracontractive for every~$p\geq2$, that is, if there exists some time~$T_p$ such that
    \begin{equation}
        \|e^{-T_p\mathcal{L}} u_0\|_\infty \leq A_p(T_p) \|u_0\|_p,
    \end{equation}
    then we also have that~$\mathcal{L}$ is~$1$--hypercontractive, that is,
    \begin{equation}
        \|e^{-T_p\mathcal{L}} v_0\|_{p'} \leq A_p(T_p) \|v_0\|_1.
    \end{equation}
\end{Proposition}

\begin{proof}
    Let~$u_0\in L^2(\RR^N) \cap L^p(\RR^N)$. Since,~$e^{-t_p^q\mathcal{L}}u_0\in L^q(\RR^N)$,  the product~$\langle e^{-t_p^q\mathcal{L}}u_0, v_0 \rangle$ is well-defined for any~$v_0 \in L^2(\RR^N) \cap L^{q'}(\RR^N)$. Using the self-adjointness and the hypercontractivity we get
    \begin{equation}
        \langle u_0, e^{-t_p^q\mathcal{L}}v_0 \rangle = \langle e^{-t_p^q\mathcal{L}}u_0, v_0 \rangle \leq A_p(t_p^q) \|u_0\|_p \|v_0\|_{q'}.
    \end{equation}
    We now ignore the~$L^2$ assumptions on both~$u_0$ and~$v_0$ via approximation and then taking the supremum of~$u_0$ functions in~$L^p$ with~$\|u_0\|_p = 1$ we obtain the first result. For the~\mbox{$1$--hypercontractivity} result, we just follow the same method.
\end{proof}

Now let us give the integral condition that characterizes the strong hypercontractivity.

\begin{Proposition}\label{re:logtostronghyper}
    Let~$A_p$ satisfy the hypotheses of~\Cref{re:hypertolog} and assume~$q$ is given by the solution of
    \begin{equation}
        q'(t) = h(q(t)),\qquad q(0)=p
    \end{equation}
    for some continuous positive function~$h:[p,\infty) \to \RR$. If~$q(t)$ blows up at~$t = t_\infty < \infty$ then the~$p$--logarithmic Sobolev inequalities~\eqref{eq:entropy-ineq} are satisfied for any~$p\geq 2$ and~$D(p) = \frac{p^2}{h(p)}$ verifying
    \begin{equation} \label{eq:explosioncondition}
        t_\infty = \int_p^\infty \frac{D(\alpha)}{\alpha^2} \, d\alpha < \infty.
    \end{equation}
    Reciprocally, under the hypotheses of~\Cref{re:logtohyper}, the operator~$\mathcal{L}$ is strongly~$p$--hypercontractive for this~$t_\infty$.
\end{Proposition}
\begin{proof}
    The function~$q$ solution of the ODE satisfies the conditions of~\Cref{re:hypertolog}. Condition~\eqref{eq:explosioncondition} is nothing but the standard blow-up condition for ODEs (also known as Osgood's condition) written in terms of~$D$ instead of~$h$. 
\end{proof}

This result characterizes strong~$p$--hypercontractivity under the mild hypothesis that the function~$q$ appearing in the definition of~$p$--hypercontractivity~\eqref{def:hypercontractive} comes from an autonomous ODE. Notice that for the integral condition~\eqref{re:logtostronghyper}, only the behaviour as~$p\to\infty$ of~$D(p)$ matters. Adding the same integral condition but for~$C(p)$ instead allows us to prove the eventual ultracontractivity of the operator.

\begin{Corollary}\label{re:logtoeventualultra}
    Under the hypotheses of~\Cref{re:logtohyper}, the operator~$\mathcal{L}$ is eventually ultracontractive if, for any~$p \geq 2$,
    \begin{equation}\label{eq:explosionandboundednesscondition}
        \int_p^\infty\frac{D(\alpha)}{\alpha^2}\,d\alpha<\infty\quad\text{and }\int_p^\infty\frac{C(\alpha)}{\alpha^2}\,d\alpha<\infty.
    \end{equation}
\end{Corollary}

\begin{proof}
    Thanks to the first integral condition, we know that there exists a time~$t_\infty$ such that~$u(t)\in L^q(\RR^N)$ for all~$p\in[1,\infty)$ and every~$t \geq t_\infty$. The second integral condition implies that~$A_p(t_\infty) < \infty$. Then, for every~$t\geq t_\infty$,
    \begin{equation}
        \|u(t)\|_q \leq \sup_{0<t \leq t_\infty} A_p(t) \|u_0\| \quad \text{for all } q \geq p,
    \end{equation}
    and this uniform bound implies that~$u(t)$ is bounded for every~$t\geq t_\infty$, with the same bound.
\end{proof}
Observe that this result does not yield a complete characterization of eventual ultracontractivity. Indeed, applying~\Cref{re:hypertolog} to an eventually ultracontractive operator~$\mathcal{L}$ does not provide the integral condition on~$C$ required in~\Cref{re:logtoeventualultra} to obtain ultracontractivity, even under the assumption that~$q$ arises from an ODE.

\begin{Remark}
    The function~$D$ is the only responsible for the growth of the integrability exponent~$q(t)$. For a fixed~$t \geq 0$,~$q(t)$ is then fixed, and the only responsible for the control of the norm~$\|u(t)\|_{q(t)}$ is the function~$C(p)$. Notice that, in both cases, only the behaviour of~$D(p)$ and~$C(p)$ when~$p \gg 1$ matters.
\end{Remark}

Let~$D$ be constant, which implies the existence of some~$0<t_\infty<\infty$ such that~$\lim_{t \to t_\infty} q(t) = \infty$. If also~$C$ is constant, then
\begin{equation}
    \int_p^{\infty} \frac{C}{\alpha^2} \, d\alpha = \frac Cp < \infty,
\end{equation}
and thus we obtain eventual ultracontractivity. However, the situation that we will find throughout the paper is more similar to~$C(p) = p$, which implies that
\begin{equation}
    \int_p^{q(t)} \frac{C(\alpha)}{\alpha^2} \, d\alpha = \log\left(\frac{q(t)}{p} \right),
\end{equation}
and therefore
\begin{equation}
    \|u(t)\|_{q(t)} \leq A_p(t) \|u_0\|_p, \quad \text{where } A_p(t) = \frac{q(t)}{p}.
\end{equation}
Hence,~$\lim_{t \to t_\infty} A_p(t) = \infty$. This means that even though~$q(t)$ blows up at~$t_\infty$, we cannot deduce from this that~$u$ becomes bounded. In fact, along the paper we will find that solutions in general will not become bounded at~$t=t_\infty$. Nevertheless, they will become bounded for~$t > t_\infty$, at least in the translation-invariant setting.

Now, let us show that we can understand supercontractivity and ultracontractivity within the hypercontractivity framework, through the same Stein's interpolation argument that is used in~\cite{DaviesBook, DaviesSimonIntrinsic} to show ultracontractivity. This allows us to fully characterize supercontractivity, while giving the expected sufficient conditions for ultracontractivity. Observe that the condition enabling the passage from supercontractivity to ultracontractivity is analogous to the one that allows us to pass from strong hypercontractivity to eventual ultracontractivity.
\begin{Theorem} \label{re:supercontractivity}
    Let~$e^{-t\mathcal{L}}$ be a symmetric Markov semigroup.
    \begin{itemize}
        \item Let~$q\geq2$. If~$\mathcal{L}$ is supercontractive, i.e., if it satisfies~\eqref{def:supercontractive}, then the logarithmic Sobolev inequality         
        \begin{equation} \label{eq:logdelsuper}
                    \ent(f^2) \leq 2\beta(\varepsilon) \|f\|_2^2 + 2\varepsilon \overline{\mathcal{E}}(f), \qquad \beta(\varepsilon) = (1-2/q)^{-1} \log\Big(B\big(\varepsilon (1-2/q),q\big)\Big),
        \end{equation}
        holds for any~$\varepsilon > 0$ and any~$f \in L^1(\RR^N) \cap H_\mathcal{L}(\RR^N)$.
        \item If~$\mathcal{L}$ is ultracontractive, i.e., if it satisfies~\eqref{def:ultracontractive}, then~\eqref{eq:logdelsuper} holds with~$\beta(\varepsilon) = \log B(\varepsilon)$.
    \end{itemize}
    On the other hand:
    \begin{itemize}
        \item If~$\mathcal{L}$ satisfies a logarithmic Sobolev inequality as~\eqref{eq:logdelsuper} for all~$\varepsilon>0$ with~$\beta \colon (0, \infty) \to \RR$ any continuous function, then~$\mathcal{L}$ is supercontractive.
        \item If, in addition, 
            \begin{equation}
                M(t) = \frac1{t} \int_0^{t} \beta(s) \, ds < \infty\quad\textup{for any }t>0,
            \end{equation}
            then~$\mathcal{L}$ is ultracontractive.
    \end{itemize}
\end{Theorem}
\begin{proof}
Fix~$t>0$ to be determined later and take~$q>0$. The operator defined as~$P(s) = e^{-s\mathcal{L}}$ for~$0 \leq s < t$ satisfies
  $$
  \begin{array}{ll}
  P(0) \colon L^2(\RR^N) \to L^2(\RR^N),&\quad \|P(0)\|_{2,2} = 1, \\ [3mm]
  P(t) \colon L^2(\RR^N) \to L^q(\RR^N),&\quad \|P(t)\|_{q,2} \leq B(t,q).
  \end{array}
  $$
  Then, through Stein's interpolation theorem we obtain that~$P(s): L^2(\RR^N) \to L^{r(s)}(\RR^N)$ with
    \begin{equation}
        \frac1{r(s)} = \frac{1-s/t}2+\frac{s/t}{q}\;\rightarrow\;r(s)= \frac{2qt}{qt - s(q-2)},
    \end{equation}
  and
  $$
   \|P(s)\|_{r(s),2} \leq B(t,q)^{s/t}.
  $$
  Hence,
    \begin{equation}
        \|u(t)\|_{r(s)} \leq A_2(s) \|u_0\|_2,\quad \text{where } A_2(s) = B(t,q)^{s/t}.
    \end{equation}
    Notice that we have translated the supercontractivity to a strong hypercontractivity that explodes at this fixed time~$t>0$. Then we can apply~\Cref{re:hypertolog} to obtain
 $$
 C(2)=\dfrac{4A_2'(0)}{r'(0)}=\dfrac{2q}{q-2}\log(B(t,q)),\qquad
 D(2)=\dfrac{4}{r'(0)}=\dfrac{2qt}{q-2},
 $$    
and choosing $t=(q-2)\varepsilon/q$ get~\eqref{eq:logdelsuper}. Ultracontractivity follows by letting~$q \to \infty$.

    For the reciprocal result, remember that from a~$2$--logarithmic Sobolev inequality we can obtain a family of~$p$--logarithmic Sobolev inequalities as~\eqref{eq:badlogfamily}, which are usually not optimal for hypercontractivity. However, we have the freedom to choose the parameter~$\varepsilon>0$. Given~$p\geq 2$, the coefficient appearing in front of the energy term satisfies
    \begin{equation}
        \varepsilon\frac{p^2}{2(p-1)} \leq \varepsilon p.
    \end{equation}
    Take~$\varepsilon = \frac{2t}{p}$. This implies strong hypercontractivity for
    \begin{equation}
        t_\infty = 2t \int_2^\infty \frac{1}{\alpha^2}\,d\alpha = t,
    \end{equation}
    thanks to~\Cref{re:logtohyper,re:logtostronghyper}. Since~$t>0$ is arbitrary, we conclude the result. 
    
    To prove ultracontractivity, applying~\Cref{re:logtoeventualultra}, we define~$C(t,p) = 2\beta\left(\frac{2t}{p} \right)$. We just need~$C$ to satisfy the condition
    \begin{equation} \tag*{ }
        M(t) = \int_2^\infty \frac{C(t,\alpha)}{\alpha^2} \, d\alpha = \frac1{t} \int_0^{t} \beta(s) \, ds < \infty \quad \text{for } t>0.\qedhere
    \end{equation}   
\end{proof}

\smallskip

\begin{Remark}
    While supercontractivity is fully characterized in the previous result, ultracontractivity is not: when going from the logarithmic Sobolev inequality to ultracontractivity, an integral condition appears. This integral condition is not recovered in the reciprocal result, and thus the equivalence fails.
\end{Remark}

To conclude this section, let us examine what happens when an operator~$\mathcal{L}$ is known to be hypercontractive, but we attempt to rederive its hypercontractivity by first applying~\Cref{re:hypertolog} to obtain its logarithmic Sobolev inequalities, and then using~\Cref{re:logtohyper} to return to hypercontractivity.

Fix~$p\geq2$. Assume that~$\mathcal{L}$ is~$r$--hypercontractive for all~$r\geq p$, i.e.,
\begin{equation}
    \|u(t)\|_{q(t)} \leq A_r(t) \|u_0\|_r \quad \text{for all } r \geq p.
\end{equation}
Under the hypotheses of~\Cref{re:hypertolog}, one obtains the~$r$--logarithmic Sobolev inequalities
\begin{equation}
    \ent(|f|^r) \leq \frac{r^2 A_r'(0)}{q'(0)} \|f\|_r^r + \frac{r^2}{q'(0)} \mathcal{E}(|f|^{r-2}f,f) \quad \text{for all } r \geq p
\end{equation}
for all admissible~$f$. Now, if one then uses~\Cref{re:logtohyper}, since~$C(r) = \frac{r^2 A_r'(0)}{q'(0)}$ and~$D(r) = \frac{r^2}{q'(0)}$, we find that there exists some other~$\widetilde{A}_p(t)$ such that
\begin{equation}
    \widetilde{A}_p(t)=\exp\left(\int_0^t A'_{q(s)}(0)\,ds\right)=\exp\left(\int_p^{q(t)}\frac{A'_{\alpha}(0)}{\alpha^2}\,d\alpha\right).
\end{equation}
This formula gives us in principle a worse function~$\widetilde{A}_p$ than the~$A_p$ we initially started with. There is a nice formal way to give this integral expression some intuition. If~$A_p$ is the optimal bound for hypercontractivity, observe that~$A_p(t_p^q) \leq A_r(t_r^q) A_p(t_p^r)$, for~$p\leq r \leq q$. Then,
\begin{equation}
    A_p(t_p^r + h) - A_p(t_p^r) \leq A_r(h) A_p(t_p^r) - A_p(t_p^r) = A_p(t_p^r) \big(A_r(h) -1\big).
\end{equation}
Now, dividing by~$h>0$ and taking the limit~$h \to 0$, we find that
\begin{equation}
    \frac{A_p'(t_p^r)}{A_p(t_p^r)} \leq A'_r(0).
\end{equation}
Time~$t_p^r = t(r)$ is a function of~$r$. Inverting the function to obtain~$r(t)$ and integrating,
\begin{equation}
    A_p'(t) \leq \exp\left( \int_0^t A'_{r(s)}(0) \, ds \right),
\end{equation}
just as we wanted to show.

The hypercontractivity loop that we have just established is merely illustrative if hypercontractivity is already known, but it is one of the key ideas of this paper, to be fully developed in~\Cref{se:generaloperators}. Assume that there are two operators defined by comparable L\'evy kernels~$J_1$ and~$J_2$, and let the first operator be hypercontractive. There is no apparent way to transfer this hypercontractivity to the second operator. However, it is straightforward to check that if the first operator satisfies a~$p$--logarithmic Sobolev inequality, then the second one will too, and hypercontractivity will follow. This loop method will also be employed to obtain supercontractivity.

\section{The~\texorpdfstring{$\log(I-\Delta)$}{log(I − ∆)} operator} \label{se:logoperator}

To apply the results of the previous section, we will work with $0^+$-order operators. The idea is to work with an operator~$\mathcal{L}$ with symbol~$m(\xi) \sim \log(\xi)$ for~$\xi \sim \infty$, since these are the kind of operators expected to have hypercontractivity, but not ultracontractivity; see~\Cref{se:convolution}. If we use~$m(\xi) = \log(1+|\xi|^2)$, we find the model operator for this work,~$\mathcal{L}=\log(I-\Delta)$. This operator has the advantage of being a subordinated operator to the Laplacian. It is defined~as
\begin{equation}
[\log(I-\Delta)u]\,\widehat{\;}\,(\xi)=\log(1+|\xi|^2)\widehat{u}(\xi).
\end{equation}
It also appears as the infinitesimal generator of the Bessel potential~$(I-\Delta)^{-t}$ semigroup, i.e.,
\begin{equation}
    -\log(I-\Delta) = \lim_{t \to 0^+} \frac{(I-\Delta)^{-t} -I}{t},
\end{equation}
which will be of importance in order to understand the gradual improvement of solutions.

\subsection{Representation in terms of a Lévy kernel} The starting point to express the operator in terms of a Lévy kernel is the following computation, 
$$
    \begin{array}{rl}
    \displaystyle\int_0^\infty(1-e^{-st})\frac{e^{-t}}t\,dt &\displaystyle=\int_0^\infty\frac{e^{-t}}t\int_0^t se^{-s\tau }\,d\tau dt= s\int_0^\infty e^{-s\tau}\int_\tau^\infty\frac{e^{-t}}t\,dt d\tau\\[3mm]
    &\displaystyle=s\int_0^\infty e^{-s\tau}\int_0^1\frac{e^{-\tau/z}}z\,dzd\tau=s\int_0^1\frac1z\int_0^\infty e^{-\tau(s+1/z)}\,d\tau dz\\[3mm]
    &\displaystyle=s\int_0^1\frac1{z(s+1/z)}\,dz=\log(1+s).
    \end{array}
$$
Therefore,~$\displaystyle\log(1+|\xi|^2)=\int_0^\infty(1-e^{-t|\xi|^2})\frac{e^{-t}}t\,dt$; that is, the symbol of~$\log(I-\Delta)$ coincides with that of~$\displaystyle\int_0^\infty(I-e^{t\Delta})\frac{e^{-t}}t\,dt$, whence
\begin{equation} \label{eq:logsubordination}
    \log(I-\Delta)=\int_0^\infty(I-e^{t\Delta})\mathcal{I}(t)\,dt,\qquad \mathcal{I}(t)=\frac{e^{-t}}t,
\end{equation}
which is the spectral representation of~$\Phi(-\Delta)$ when~$\Phi(s)=\log(1+s)$; see~\cite{Bochner49,PotentialAnalysis,SchillingSubordinationBochner}. 

From~\eqref{eq:logsubordination} we get, for smooth functions~$f$, the representation
$$
    \displaystyle \log(I-\Delta)f(x)=\int_{\mathbb{R}^{N}}(f(x)-f(y))J(x-y)\,dy,\qquad\displaystyle J(z)=\int_0^\infty G_t(z)\mathcal{I}(t)\,dt,
$$
where~$G_t$ is the Gauss kernel. 
Thus,
\begin{equation}\label{eq:levy-kernel}
    J(z)= \frac 1{(4\pi)^\frac N2}\int_0^\infty t^{-\frac N2 - 1}e^{-\frac{|z|^2}{4t}} e^{-t} \,dt.
\end{equation}
It will be important to understand the asymptotic behaviour of~$J$, particularly at the origin.

\begin{Proposition}
    For~$J$ defined in~\eqref{eq:levy-kernel}, we have the following asymptotic behaviour
    \begin{equation}
        \lim_{|z|\to 0}|z|^N J(z) = \frac{\Gamma(N/2)}{\pi^{N/2}}, \qquad \lim_{|z| \to \infty}|z|^{\frac{N+1}2}e^{|z|} J(z) = (2\pi)^{-\frac{N-1}2}.
    \end{equation}
\end{Proposition}

\begin{proof} Performing the change of variables~$s=\frac{1}{4t}$ in~\eqref{eq:levy-kernel}, we obtain
    \begin{equation}
        J(z)=\frac 1{(4\pi)^{\frac N2}|z|^{N}}\int_0^\infty e^{-s|z|^2}e^{-\frac1{4s}}s^{-\frac N2-1}\,ds.
    \end{equation}
    From here, the behaviour at the origin is immediate, since
    \[
        \int_0^\infty e^{-\frac1{4s}}s^{-\frac N2-1}\,ds=4^{N/2}\int_0^\infty e^{-\tau}\tau^{\frac{N}{2}-1}\,d\tau=4^{N/2}\Gamma(N/2).
    \]
    
    On the other hand, using Laplace's method (which estimates integrals via a Taylor's expansion of the integrand; see for instance~\cite{Olver74}), we can calculate, for large~$|z|$,
   $$
        \begin{array}{rll}
        J(z)&=\displaystyle(4\pi)^{-\frac N2}\int_0^\infty t^{-\frac N2-1}e^{-\frac{|z|^2}{4t}-t}\,dt=\int_0^\infty g_1(t)e^{-g_2(z,t)}\,dt\\[3mm]
        &\displaystyle= g_1(t_0)e^{-g_2(z,t_0)}\left(\int_0^\infty e^{-\frac12\frac{\partial^2g_2}{\partial t^2}(z,t_0)s^2}\,ds+ o(1)\right)\\[3mm]
        &\displaystyle=g_1(t_0)e^{-g_2(z,t_0)}\left(\left(\frac{2\pi}{\frac{\partial^2g_2}{\partial t^2}(z,t_0)}\right)^{1/2}+o(1)\right),
        \end{array}
   $$
    where~$t_0=t_0(z)$ is the minimum in~$t$ of~$g_2$. In our case~$g_2(z,t)=\dfrac{|z|^2}{4t}+t$, with minimum at~$t_0=|z|/2$. Thus,
    \begin{equation} \tag*{ }  
        |z|^{\frac{N+1}2}e^{|z|}J(z)=\left( (2\pi)^{-\frac{N-1}2}+ o(1)\right)\quad\text{as } |z|\to\infty. \qedhere
    \end{equation}
   \end{proof}
The behaviour~$J(z)\sim|z|^{-N}$ for~$|z|\sim0$ shows that this is a $0^+$-order operator. This fact is of great importance, since~$\log(I-\Delta)$ will be the model for all operators whose kernel behaves that way. The behaviour at infinity is also studied in~\cite{Sikic-Song-Vondracek-2006}. In particular the integral is a representation of a Bessel function of second order.

\subsection{The heat kernel} 

If~$u$ is the solution to equation~$\partial_t u + \log(I-\Delta)=0$ with~$u(0)=u_0 \in L^p(\RR^N)$, we find that
\begin{equation}
    \widehat{u}(t) = e^{-t\log(1+|\xi|^2)}\widehat{u_0}=(1+|\xi|^2)^{-t} \widehat{u_0}.
\end{equation}
As we showed in~\Cref{se:preliminaries}, this implies that the solution~$u(t)$ can be written as
\begin{equation} \label{eq:besselreprsol}
    u(t) = (I-\Delta)^{-t} u_0,
\end{equation}
with~$(I-\Delta)^{-t}$ the Bessel potential. It also admits the expression
\begin{equation}
  (I-\Delta)^{-t} u_0 = H_{t}*u_0,
\end{equation}
where~$H_t$ is the inverse Fourier transform of
\begin{equation} \label{eq:previoussubordination}
    \widehat{H_t}(\xi)=e^{-tm_0(\xi)}=(1+|\xi|^2)^{-t},
\end{equation}
that is,~$H_t$ is the fundamental solution associated to~$\log(I-\Delta)$. It can also be computed explicitly thanks to the subordination formula
\begin{equation}
    H_t(x)=\int_0^\infty G_\tau(x)\Psi_t(\tau)\,d\tau,\qquad \Psi_t(\tau)=\dfrac{\tau^{t-1}e^{-\tau}}{\Gamma(t)}.
\end{equation}
The proof is straightforward using the Laplace transform~${\widetilde \Psi}_t$ of~$\Psi_t$,
\[
    \widehat{H_t}(\xi)=\int_0^\infty e^{-\tau|\xi|^2}\Psi_t(\tau)\,d\tau={\widetilde \Psi}_t(|\xi|^2)=\frac1{(1+|\xi|^2)^t}.
\]
The function~$\Psi_t$ is called in the literature the \emph{transition function} corresponding to~$\Phi$, and formula~\eqref{eq:previoussubordination} is called a \emph{subordination} between~$G_t$ and~$H_t$.
Observe also that~$\displaystyle\int_0^\infty \Psi_t(\tau)\,d\tau={\widetilde \Psi}_t(0)=1$. In particular, since clearly~$H_t>0$, we have~$\|H_t\|_1=1$ for every~$t>0$. To sum up,
\begin{equation}\label{eq:fundamentalsolution}
    H_t(x)=\frac 1{(4\pi)^{N/2}\Gamma(t)}\int_0^\infty s^{t-\frac N2-1}e^{-s}e^{-\frac{|x|^2}{4s}}\,ds.
\end{equation}
This is a well-known formula that appears in several books, for instance~\cite{GrafakosModern, Steinbook}.

\subsection{Smoothing properties}

The above discussion gives a lot of information, which we summarize next.

\begin{Proposition} \label{re:solutionforlog}
    Let~$u(t)= H_t * u_0$ be the solution to problem~\eqref{problem-L} for~$\mathcal{L} = \log(I-\Delta)$ and~$u_0 \in L^p(\RR^N)$ for~$1<p<\infty$. Then~$u(t) = (I-\Delta)^{-t} u_0  \in L^p(\RR^N) \cap L^p_{2t}(\RR^N)$, and
    \begin{equation}
        \|u(t)\|_{L^p_{2t}} = \|u_0\|_{L^p}, \qquad \|u(t)\|_{L^p} \leq \|u_0\|_{L^p}.
    \end{equation}
    Thus, we have that
    \begin{equation}
        u(t) \in L^p(\RR^N) \cap
        \begin{cases}
            L^{\frac{Np}{N-2pt}}(\RR^N),&0<t<\frac{N}{2p}, \\
            L^q(\RR^N) \text{ for all } q\in[p,\infty),& t = \frac{N}{2p}, \\
            \cont_0(\RR^N),&t > \frac{N}{2p}.
        \end{cases}
    \end{equation}
    Given~$q\geq p$, the time required for the solution~$u$ to get into~$L^q(\RR^N)$ is
    \begin{equation}
        t_p^q = \frac{N}{2}\left( \frac 1p - \frac 1q \right).
    \end{equation}
\end{Proposition}

\begin{proof}
    Since~$u_0 \in L^p(\RR^N)$, by definition~$u(t) = (I-\Delta)^{-t}u_0$ belongs to the fractional Bessel potential space~$L^p_{2t}(\RR^N)$. Moreover, one can easily check that
    \begin{equation}
        \|u(t)\|_{L^p_{2t}} = \|(I-\Delta)^t u(t)\|_{L^p} = \|u_0\|_p.
    \end{equation}
    Now, from the embeddings for these spaces, see~\cite{Adamsbook,AdamsFourierbook,GrafakosModern}, we deduce that
    \begin{equation}
        u(t) \in L^p_{2t}(\RR^N) \implies u(t) \in
        \begin{cases}
            L^{\frac{Np}{N-2pt}}(\RR^N), & 0<t<\frac{N}{2p}, \\
            L^q(\RR^N) \text{ for all } q\in[p,\infty),&t = \frac{N}{2p}, \\
            \cont_0(\RR^N),&t > \frac{N}{2p}.
        \end{cases}
    \end{equation}
    From the fact that~$\|H_t\|_1 = 1$ and Young's convolution inequality we obtain that~$\|u(t)\|_p \leq \|u_0\|_p$. Finally, from~$q = \frac{Np}{N-2p t_p^q}$ we deduce the expression for~$t_p^q$.
\end{proof}

This result not only establishes hypercontractivity for $\log(I-\Delta)$, but also provides strong hypercontractivity at time $t=\frac{N}{2p}$ and eventual ultracontractivity for any $t>\frac{N}{2p}$.

\begin{Remark}
    The case of integrable initial data~$p = 1$ is absent from the previous result, since the corresponding space~$L^{1}_{2t}$ for the Bessel potential is usually avoided in the literature. In the study of these regularity spaces, it is customary to replace~$L^{1}$ with the corresponding real Hardy space, as it interacts better with harmonic-analysis techniques. Nevertheless, it is possible to prove a slightly different yet analogous result to the one above for~$p = 1$; see~\Cref{re:solutionforlogL1}.
\end{Remark}

\begin{Remark}\label{rk:higher.regularity}
    For~$t > \frac N{2p}$ one can use Morrey-type embeddings to show that solutions get progressively more regular. Indeed, for any given~$k>0$, there exists~$t_k>0$ such that~$u(t)\in\cont^k(\RR^N)$ for all~$t > t_k$.
\end{Remark}

Observe that we can talk about an improvement of the regularity of the solution over time thanks to the fact that~$u(t)$ stays in the same~$L^p$ space as the initial data~$u_0$ belonged to. Otherwise the space to which the solution belongs to is changing but not necessarily improving, as is the case for solutions to~$\partial_t u + \log(-\Delta)u = 0$ for data~$u_0 \in L^p(\RR^N)$, see~\Cref{subse:badlog}. Notice that this improvement occurs gradually over time; in particular, there is no supercontractivity. The embeddings for the Bessel potential used in the previous result are optimal.

Let us focus on the time interval~$0<t<\frac{N}{2p}$. For our purpose of obtaining a family of good logarithmic Sobolev inequalities, knowing that~$u(t)$ belongs to~$L^{\frac{Np}{N-2pt}}(\RR^N)$ is not enough. We require a precise control of its norm, in particular its dependence on time; that is, we need to understand the function~$A_p(t)$ satisfying the estimate
\begin{equation}
    \|u(t)\|_{\frac{Np}{N-2pt}} \leq A_p(t) \|u_0\|_p.
\end{equation}
Specifically, in order to apply~\Cref{re:hypertolog}, we need to show that~$A_p(0) = 1$, and that~$A'_p(0)$ exists. For a fixed~$t>0$, the asymptotic behaviour for the fundamental solution~\eqref{eq:fundamentalsolution} appears in~\cite{Nikolskiibook,Watsonbook,Titchmarshbook},
\begin{equation} \label{eq:twotermasymptotics}
    H_t(x) = \frac{\Gamma(\frac{N}{2}-t)}{4^{t}\pi^{N/2} \Gamma(t)} \frac{1}{|x|^{N-2t}} + O(|x|^{-N+2t+2}).
\end{equation}
From this it is easy to observe that~$H_t \in L^{\frac{N}{N-2t}, \infty}(\RR^N)$, since~$\||x|^{-N+2t}\|_{\frac{N}{N-2t}, \infty} = \left(\frac{\omega_{N-1}}N \right)^{\frac{N-2t}N}$, where $\omega_{N-1}$ is the measure of the unit sphere in $\RR^N$. With this information we may deduce the Sobolev embeddings found in~\Cref{re:solutionforlog}. We may also deduce a similar result for~$L^1$ initial data.

\begin{Proposition} \label{re:solutionforlogL1}
    Let~$u$ be the solution of~\eqref{problem-L} for~$\mathcal{L} = \log(I-\Delta)$ with~$u_0 \in L^1(\RR^N)$. Then~$\|u(t)\|_1 = \|u_0\|_1$ and
    \begin{equation}
        u(t) \in L^1(\RR^N) \cap
        \begin{cases}
            L^{\frac{N}{N-2t}, \infty}(\RR^N),&0<t<\frac{N}{2}, \\
            L^q(\RR^N) \text{ for all } q\in[1,\infty),&t = \frac{N}{2}, \\
            \cont_0(\RR^N),&t > \frac{N}{2}.
        \end{cases}
    \end{equation}
\end{Proposition}

\begin{proof}
    For the claim that~$\|u(t)\|_1 = \|u_0\|_1$, we simply use that~$u(t) = H_t * u_0$ and that~$\|H_t\|_1 = 1$. To show that~$u(t) \in L^{\frac{N}{N-2t}, \infty}(\RR^N)$ for~$0<t<\frac{N}{2}$, we use weak-weak Young's inequality, see~\cite{GrafakosClassical}, and the fact that~$H_t \in L^{\frac{N}{N-2t}, \infty}(\RR^N)$. For~$t = \frac{N}{2}$, we employ Young's inequality and that~$H_t \in L^q(\RR^N)$ for all~$1 \leq q < \infty$. Lastly, for~$t>\frac N2$ we use that~$H_t \in L^\infty(\RR^N)$.
\end{proof}

Even though~\eqref{eq:twotermasymptotics} was useful for this qualitative result, there are hidden factors depending on~$t$ in its last term. Thus, the above formula is not enough to get precise quantitative estimates. Instead, we use a simpler but, for our purposes, much more powerful approach. In~\eqref{eq:fundamentalsolution} we perform the change of variables~$\tau = \frac{|x|^2}{4s}$ , which yields
\begin{equation} \label{eq:fundamentalchangeofvar}
    H_t(x) = \frac{1}{|x|^{N-2t}} \frac{1}{4^{t}\pi^{N/2}\Gamma(t)} \int_0^\infty \tau^{\frac{N}{2}-t -1}e^{-\tau}e^{-\frac{|x|^2}{4\tau}}\,d\tau.
\end{equation}
Thanks to this formulation, we prove the result below.

\begin{Proposition} \label{re:weakfundamentalnorm}
    Let~$H_t$ be as in~\eqref{eq:fundamentalsolution}. Then, if~$t<N/2$,
    \begin{equation} \label{eq:fundamentalweak}
        \|H_t\|_{\frac N{N-2t},\infty}=\left(\frac{\omega_{N-1}}N\right)^{\frac{N-2t}N}\frac{\Gamma\left(\frac N2-t\right)}{4^{t}\pi^{N/2}\Gamma(t)}.
    \end{equation}
\end{Proposition}

\begin{proof}
    Recall that
    \begin{equation} \label{eq:weaknorm}
        \|f\|_{p, \infty} = \sup_{\lambda>0} \lambda (\mu\{f(x)> \lambda\})^{\frac 1p}.
    \end{equation}
    The function
    \begin{equation}
        g(r) = \int_0^\infty \tau^{\frac{N}{2}-t -1}e^{-\tau}e^{-\frac{r^2}{4\tau}}\,d\tau,\quad r \geq 0,
    \end{equation}
    is bounded, since~$t<N/2$, and decreases as~$r$ increases. Then,
    \begin{equation}\label{eq:sup.g}
        \sup_{r>0} g(r) = g(0) = \Gamma\left( \frac N2 - t \right).
    \end{equation}
    This implies, using formula~\eqref{eq:fundamentalchangeofvar}, that
    \begin{equation}
        H_t(x) \leq R_t(x) = \frac 1{|x|^{N-2t}} \frac{{\Gamma\left( \frac N2 - t \right)}}{4^{t}\pi^{N/2}\Gamma(t)},
    \end{equation}
    where~$R_t$ is the Riesz potential. For a given~$\lambda > 0$, observe that
    \begin{equation}
        \{x : H_t(x) > \lambda\} \subset \{x : R_t(x) > \lambda\},
    \end{equation}
    concluding from definition~\eqref{eq:weaknorm} that
    \begin{equation}
        \|H_t\|_{\frac N{N-2t}, \infty} \leq \|R_t\|_{\frac N{N-2t}, \infty} = \left(\frac{\omega_{N-1}}N \right)^{\frac{N-2t}N} \frac{\Gamma\left( \frac N2 - t \right)}{4^{t}\pi^{N/2}\Gamma(t)}.
    \end{equation}
    On the other hand, notice from definition~\eqref{eq:weaknorm} that
    \begin{equation}
        \lim_{\lambda \to \infty} \lambda (\mu\{f(x)> \lambda\})^{\frac 1p} \leq \|f\|_{p, \infty}.
    \end{equation}
    Since~$H_t(x)$ strictly decreases with~$|x|$, it is easy to see that
    \begin{equation}
        \mu\{H_t(x) > \lambda\} = \left(\frac{\omega_{N-1}}N \right) \big( (H_t)^{-1}(\lambda)\big)^N,
    \end{equation}
    where~$(H_t)^{-1}$ is the inverse function of~$H_t$. Notice that, by~\eqref{eq:sup.g},
    \begin{equation}
        \lim_{|x|\to 0} |x|^{N-2t} H_t(x) = \frac{{\Gamma\left( \frac N2 - t \right)}}{4^{t}\pi^{N/2}\Gamma(t)},
    \end{equation}
    whence
    \begin{equation}
        \lim_{\lambda \to \infty} \lambda^{1/(N-2t)} (H_t)^{-1}(\lambda) = \left(\frac{{\Gamma\left( \frac N2 - t \right)}}{4^{t}\pi^{N/2}\Gamma(t)} \right)^{1/(N-2t)}.
    \end{equation}
    We conclude that
    \begin{equation}
        \|H_t\|_{\frac N{N-2t}, \infty} \geq \lim_{\lambda \to \infty}  \lambda \left(\frac{\omega_{N-1}}N \right)^{\frac{N-2t}N} \big( (H_t)^{-1}(\lambda)\big)^{N-2t} = \left(\frac{\omega_{N-1}}N \right)^{\frac{N-2t}N}  \frac{\Gamma\left( \frac N2 - t \right)}{4^{t}\pi^{N/2}\Gamma(t)},
    \end{equation}
    which gives the exact formula~\eqref{re:weakfundamentalnorm}.
\end{proof}

\subsection{Quantitative hypercontractivity}

With the previous results at hand, we are now prepared to study the hypercontractivity of our operator in a quantitative manner. Observe that understanding the hypercontractivity of~$\log(I-\Delta)$ is essentially equivalent to understanding the constants in the Sobolev embeddings associated with the Bessel potential.

The next result follows from a careful application of weak Young’s inequality together with~\Cref{re:weakfundamentalnorm}. It relies on a refined analysis of the Sobolev embeddings associated with the Bessel potential. Although the hypercontractivity of the operator was already known, the argument presented here yields substantial quantitative improvements. 

\begin{Theorem} \label{re:loghypercontractivity}
    For any~$u_0 \in L^p(\RR^N)$, with~$p>1$, let~$u(t)=H_t*u_0$. Then,
    \begin{equation}\label{eq:estimate.hypercontractivity.log}
        \|u(t)\|_{\frac{Np}{N - 2pt}} \leq A_p(t) \|u_0\|_p \quad \text{for } 0 \leq t < \frac{N}{2p},
    \end{equation}
    with
    \begin{eqnarray}
        \label{eq:A_2}
        A_2(t) &=&\displaystyle \frac{1}{(4\pi)^t} \left(\frac{\Gamma\left(\frac{N}{2}-2t\right)}{\Gamma\left(\frac{N}{2}+2t\right)}\right)^{1/2} \left(\frac{\Gamma\left(N\right)}{\Gamma\left(\frac{N}{2}\right)}\right)^{2t/N},\\
        \label{eq:A_p}
        A_p(t) &=& \displaystyle \frac{\Gamma\left( \frac N2 - t \right)}{4^{t}\pi^{N/2}\Gamma(t)} g_p(t) \quad \text{for } p>1,
    \end{eqnarray}
    where
    \begin{equation}\label{eq:upper.bound.best.constant}
        g_p(t) = \frac{1}{{2t}}\left(\frac{\omega_{N-1}}{N}\right)^{\frac{N - 2 t}{N}} \frac{N \left(p - 1\right) + 2 p t}{p^{2}} \left(\left(\frac{p \left(N - 2 t\right)}{N - 2 p t}\right)^{\frac{N - 2 t}{N}} + \left(\frac{p \left(N - 2 t\right)}{N \left(p - 1\right)}\right)^{\frac{N - 2 t}{N}}\right).
    \end{equation}
\end{Theorem}
\begin{proof}
    Using weak Young's inequality, see~\cite{LiebLossbook}, we get
    \begin{equation} \label{eq:weakYoung}
        \|H_t * u_0\|_{\frac{Np}{N-2pt}} \leq C_{N,p}(t) \|H_t\|_{\frac N{N-2t}, \infty} \|u_0\|_p,
    \end{equation}
    where~$C_{N,p}(t)$ is the best possible constant for the inequality. We need the explicit expression of~$C_{N,p}(t)$, or at least a sufficiently good upper bound. As they explain in~\cite{LiebLossbook}, the best constant in weak Young's inequality is exactly the best constant~$c_{N, p, \lambda}$ for Hardy-Littlewood-Sobolev's inequality,
    \begin{equation}
        \left\| \frac{1}{|x|^\lambda} * f \right\|_q \leq c_{N, p, \lambda} \| f \|_p \quad \text{for } \frac 1p + \frac \lambda q = 1 + \frac 1q,
    \end{equation}
    multiplied by~$\left(\frac N{\omega_{N-1}} \right)^{1/q}$. Lieb proved in~\cite{LiebSharpconstants} that this optimal constant is indeed attained, but its explicit expression has been obtained only in a limited number of cases.

    \noindent\textsc{Case~$p=2$.} In this case the best constant is known, see~\cite{LiebSharpconstants,LiebLossbook}, hence the constant in~\eqref{eq:weakYoung} is
    \begin{equation}
        C_{N,p}(t) = \pi^{\frac{N}{2}-t} \frac{\Gamma(t)}{\Gamma\left(\frac{N}{2}-t\right)} \left(\frac{\Gamma\left(\frac{N}{2}-2t\right)}{\Gamma\left(\frac{N}{2}+2t\right)}\right)^{1/2} \left(\frac{\Gamma\left(N\right)}{\Gamma\left(\frac{N}{2}\right)}\right)^{2t/N}.
    \end{equation}
    Multiplying this by~\eqref{eq:fundamentalweak}, we obtain~\eqref{eq:A_2}.

    \noindent\textsc{General case.}
    In this case, the best constant for Hardy-Littlewood-Sobolev's inequality is not known. However, we have the upper bound~$C_{N,p}(t)\le g_p(t)$ with~$g_p$ as in~\eqref{eq:upper.bound.best.constant}; see~\cite{LiebLossbook}.
    We now multiply this estimate by~$\left(\frac N{\omega_{N-1}} \right)^{\frac{N-2t}{N}}$, and we get~\eqref{eq:A_p}.
\end{proof}

\subsection{Logarithmic Sobolev inequalities}

We have now all the ingredients needed to establish the logarithmic Sobolev inequalities for~$\log(I-\Delta)$. Let us denote by~$H_\mathcal{\log}(\RR^N)$ the energy space associated to~$\log(I-\Delta)$.

\begin{Corollary} \label{re:logdellog}
    Let~$f \in L^1(\RR^N)\cap L^p(\RR^N) \cap H_\mathcal{\log}(\RR^N)$. Then,
    \begin{equation} \label{eq:logSobolevsforlog}
        \begin{aligned}
            \ent(|f|^p)&\leq A'_p(0)\frac{N}{2}\|f\|_p^p+\frac{N}{2}\mathcal{E}_{\log}(|f|^{p-2}f,f),\quad\text{with}\\
            A'_2(0) &= -2 \psi\left( \frac{N}{2} \right) - \log{4\pi} + \frac{2}{N} \log{\left(\frac{\Gamma\left(N\right)}{\Gamma\left(\frac{N}{2}\right)}  \right)}, \\
            A'_p(0) &= \frac2N \big( p - \log{p} \big) + O(1) \quad \text{for } p \sim \infty,
        \end{aligned}
    \end{equation}
    where~$\psi = \Gamma'/\Gamma$ is the \emph{digamma function}.
\end{Corollary}
\begin{proof}
    \textsc{Case~$p=2$.} Observe that in~\eqref{eq:A_2} it holds~$A_2(0) = 1$, as needed. In order to obtain the corresponding logarithmic Sobolev inequality, we need to calculate its derivative,
\[
A_2'(t)
= A_2(t)\Big[
-\psi\!\left(\frac N2-2t\right)
-\psi\!\left(\frac N2+2t\right)
-\ln(4\pi)+\frac{2}{N}\ln\!\left(\frac{\Gamma(N)}{\Gamma(\tfrac N2)}\right)
\Big],
\]
that at zero gives 
    \begin{equation}
        A_2'(0) = -2 \psi\left( \frac{N}{2} \right) - \log{4\pi} + \frac{2}{N} \log{\left(\frac{\Gamma\left(N\right)}{\Gamma\left(\frac{N}{2}\right)}  \right)}.
    \end{equation}
    Finally, notice that for~$q(t) = \frac{2N}{N-4t}$, we have~$q'(0) = \frac{8}{N}$. Hence, using~\Cref{re:hypertolog}, the logarithmic Sobolev inequality is
    \begin{equation}
        \ent(f^2) \leq A_2'(0) \frac{N}{2} \|f\|_2^2 + \frac{N}{2} \overline{\mathcal{E}}_{\log}(f).
    \end{equation}

    \noindent\textsc{General case~$p\ge 1$.} We have to be more careful in this case, since the expression for~$A_p(t)$ is not as clean as the one for~$p=2$. From now on, we write~$\omega$ instead of~$\omega_{N-1}$. Let us call
    \begin{equation}
        F(t) = \frac{\Gamma\left( \frac N2 - t \right)}{4^{t}\pi^{N/2}\Gamma(t)},
    \end{equation}
    so that~$A_p(t) = F(t)g_p(t)$. First, notice that~$A_p(0) = 1$, since
    \begin{equation}
        \lim_{t \to 0}\frac{F(t)}t = \frac{2}{\omega}, \qquad \lim_{t \to 0} t g_p(t) = \frac{\omega}{2}.
    \end{equation}
    We now need to check that~$\lim\limits_{t \to 0} A'_p(t)$ exists. Through some work we notice that
    \begin{gather}
        \lim_{t \to 0}\frac{g'_p(t)}{\frac \alpha t + \frac{\beta}{t^{2}}} = 1, \quad\textup{with}\\
        \alpha = \displaystyle\frac{\omega}{N}\left(
        p - 1 + \frac{1}{p-1}+ \frac{1}{p}\log(p-1)
        - \log\!\left(\frac{\omega p}{N}\right)
        \right), \qquad \beta = - \dfrac{\omega}{2}.
    \end{gather}
    Since~$A_p(t) = F'(t)g_p(t) + F(t) g'_p(t)$, we only need to obtain~$F'(t) = a + b t + O(t^2)$. Using Laurent's expansion for~$\Gamma(t)$ one can check that
    \begin{equation}
        \frac{\Gamma'(t)}{\Gamma^2(t)} = -1 - 2 \gamma t + O(t^2),
    \end{equation}
where~$\gamma$ is Euler-Mascheroni's constant.  From this,
    \begin{equation}
        a = \frac{2}{\omega}, \qquad b = \frac{2}{\omega} \left( -\psi\left( \frac N2 \right) - 2 \log{4} + 2 \gamma \right).
    \end{equation}
    For~$\lim\limits_{t \to 0} A'_p(t)$ to exist, we need
    \begin{equation}
        a \frac{\omega}{2} + \beta \frac{2}{\omega} = 0,
    \end{equation}
    which can be easily checked. Finally, observe that
    \begin{equation}
        \begin{aligned}
            A'_p(0) &= \lim_{t \to 0} A'_p(t) = b \frac{\omega}{2} + \alpha \frac{2}{\omega} \\
            &= \frac{- N p^{2} \psi{\left(\frac{N}{2} \right)} - 2 N p^{2} \log{\left(2 \right)} + 2 \gamma N p^{2} + N p \psi{\left(\frac{N}{2} \right)} - 2 \gamma N p + 2 N p \log{\left(2 \right)}}{N p \left(p - 1\right)} \\
            &\quad+ \frac{2 p^{3} - 2 p^{2} \log{\left(p \right)} - 2 p^{2} \log{\left(\frac{\omega}{N} \right)} - 4 p^{2} + 4 p \log{\left(p \right)} + 2 p \log{\left(\frac{\omega}{N} \right)}}{N p \left(p - 1\right)} \\
            &\quad+ \frac{-2 p \log{\left(\frac{p}{p - 1} \right)} + 4 p - 2 \log{\left(p \right)} + 2 \log{\left(\frac{p}{p - 1} \right)}}{N p \left(p - 1\right)}.
        \end{aligned}
    \end{equation}
    Asymptotically, this behaves as
    \begin{equation}
        A'_p(0) = \frac2N \big( p - \log{p} \big) + O(1) \quad \text{for } p \sim \infty.
    \end{equation}
    For~$q(t) = \frac{Np}{N-2pt}$, we have~$q'(0) = \frac{2p^2}{N}$. Hence, using~\Cref{re:hypertolog}, we get the result.
\end{proof}

\smallskip
\begin{Remark}
    Some of the above calculations could be performed for the operator~$\log\big(I + (-\Delta)^{\sigma/2}\big)$, leading to similar properties.
\end{Remark}

\subsection{Comparison with the operator~\texorpdfstring{$\log(-\Delta)$}{log(− ∆)}}\label{subse:badlog}

Problem~\eqref{problem-L} for~$\mathcal{L} = \log(-\Delta)$ is interesting, since it behaves somewhat similarly to~$\log(I-\Delta)$, although, as we will see, with some important drawbacks. Recall that~$\log(-\Delta)$ is defined via the Fourier transform as
\begin{equation}
    [\log(-\Delta)f]\,\widehat{\;} = \log(|\xi|^2) \widehat{f}.
\end{equation}
This operator is not monotone; that is,~$\langle\log(-\Delta)f,f\rangle\geq0$ is not true in general. The logarithmic Laplacian~$\log(-\Delta)$ can be expressed as,
\begin{equation}
    \begin{gathered}
        \log(-\Delta)f(x) = c_N\int_{B_1(x)}\frac{f(x)-f(y)}{|x-y|^N}\, dy - c_N\int_{\RR^N \setminus B_1(x)}\frac{f(y)}{|x-y|^N}\, dy +  \rho_N f(x),\quad\textup{with}\\
        c_N = \pi^{-N/2} \Gamma\left(\frac N2\right), \quad  \rho_N = 2\log(2) + \psi\left(\frac N2\right) - \gamma,
    \end{gathered}
\end{equation}
see~\cite{ChenWethLog}. Solutions can be written in terms of the Riesz potential,
\begin{equation}\label{eq:rieszfundamental}
    u(t) = (-\Delta)^{-2t} u_0 = R_t * u_0,\quad\textup{where }
    R_t(x) = \frac{1}{|x|^{N-2t}} \frac{1}{4^t \pi^{N/2}} \frac{\Gamma\left( \frac N2 - t \right)}{\Gamma(t)}.
\end{equation}
Observe that~$R_t$ is not in~$L^1$, and therefore, in general,~$u(t)$ no longer belongs to the same space~$L^p(\mathbb{R}^N)$ as~$u_0$ did. However, it is clear that~$R_t$ belongs to~$L^{\frac{N}{N-2t}, \infty}(\RR^N)$. In fact, the norm can be calculated more easily than we did in~\Cref{re:weakfundamentalnorm} for~$H_t$.
\begin{Proposition}
    Let~$R_t$ be as in~\eqref{eq:rieszfundamental}. Then,
    \begin{equation}
        \|R_t\|_{\frac N{N-2t}, \infty} = \left(\frac{\omega_{N-1}}N \right)^{\frac{N-2t}N} \frac{\Gamma\left( \frac N2 - t \right)}{4^{t}\pi^{N/2}\Gamma(t)}.
    \end{equation}
\end{Proposition}
Thanks to this result, it is clear that solutions~$u(t) = R_t * u_0$ satisfy the hypercontractivity property
\begin{equation}
    \|u(t)\|_{\frac{pN}{N-2pt}} \leq  A_p(t) \|u_0\|_p
\end{equation}
with the same function coefficient~$A_p$ obtained in~\Cref{re:loghypercontractivity}. Hence, one may try to apply~\Cref{re:hypertolog} in order to obtain the family of logarithmic Sobolev inequalities for~$\log(-\Delta)$. However, the expression~\eqref{eq:logdiff} may not be well-defined for~$u(t) = R_t * u_0$, since it does not belong to the appropriate spaces. Consequently, obtaining logarithmic Sobolev inequalities ---at least by our method--- is not possible. In fact,~$u(t)$ cannot be considered a solution in the sense of~\Cref{def:L2solutions}, since~$u(t)$ may not belong to~$L^2$ for~$t>0$. It is clear that Hille-Yosida's Theorem,~\Cref{re:HilleYosida}, cannot be applied here, since~$\log(-\Delta)$ is not a monotone operator.

Even more important, solutions~$u(t) = R_t * u_0$ for~$u_0 \in L^p(\RR^N)$ cease to be well defined for~$t>\frac{N}{2p}$, even as tempered distributions, since the Fourier multiplier is not locally integrable anymore. So the main differences between problem~\eqref{problem-L} for~$\log(-\Delta)$ and~$\log(I-\Delta)$ are:
\begin{itemize}
    \item For~$p>1$ and~$0 < t < \frac{N}{2p}$, solutions of both problems satisfy
        \begin{equation}
            \|u(t)\|_{\frac{pN}{N-2pt}} \leq A_p(t) \|u_0\|_p.
        \end{equation}
    Solutions associated to~$\log(I-\Delta)$ remain in every~$L^q$ space they get in, while solutions corresponding to~$\log(-\Delta)$ only belong to~$L^{\frac{pN}{N-2pt}}$ in general.
    \item Solutions corresponding to~$\log(I-\Delta)$ keep on existing and improving for~$t > \frac{N}{2p}$, and~$u(t) \in L^p_{2t}(\RR^N)$, while solutions associated to~$\log(-\Delta)$ are no longer well-defined for~$t > \frac{N}{2p}$.
    \item The logarithmic Laplacian~$\log(-\Delta)$ cannot be understood via a subordination formula for~$-\Delta$, unlike~$\log(I-\Delta)$, since the function~$\Phi(s) = \frac{\log s}s$ is not completely monotone.
    \item While~$\log(I-\Delta)$ has a clear probabilistic interpretation, being a Lévy operator,~$\log(-\Delta)$ does not. Its fundamental solution, the Riesz potential, is not integrable, and thus cannot represent a transition density.
\end{itemize}
For more information regarding the logarithmic Laplacian~$\log(-\Delta)$, see~\cite{ChenVeronLog,ChenVeronCauchyLog,ChenWethLog,JarohsSaldanaWethLog}.

\section{General operators} \label{se:generaloperators}

Let~$\mathcal{L}$ be a purely nonlocal operator given by a general Lévy kernel~$J(x,y)$; see its definition in~\Cref{se:preliminaries}. Bear in mind that solutions to~\eqref{problem-L} exist for these operators thanks to the theory in the~\hyperref[se:existence]{Appendix}. Assume that~$J(x,y)$ is bounded from below by~$c|x-y|^{-N}$ for~$c>0$ and~$x \sim y$. For these more general operators, we cannot use the Fourier transform technique from the previous section. The alternative approach that we describe below uses all the main results from~\Cref{se:equivalence} and~\Cref{se:logoperator} of the paper.

Let us call~$J_{\log}$ the kernel of~$\log(I-\Delta)$. Thanks to~\Cref{re:hypertolog}, we know that~$\log(I-\Delta)$ satisfies the~$p$--logarithmic Sobolev inequalities~\eqref{eq:logSobolevsforlog} for all~$p\geq 2$, as we showed in~\Cref{se:logoperator}. Just like in the previous section, with a change of variables in time we can replace the constants and assume that~$J(x,y) \geq J_{\log}(x-y)$ (although it will only be required for~$x \sim y$). This condition tells us that~$J(x,y)$ is at least as singular as~$J_{\log}(x-y)$ at the diagonal. This implies that
\begin{equation}
    \mathcal{E}_{\log}(|f|^{p-2}f,f) \leq \mathcal{E}_{\mathcal{L}}(|f|^{p-2}f,f),
\end{equation}
where $\mathcal{E}_{\log}$ and $\mathcal{E}_{\mathcal{L}}$ are the respective bilinear forms of operators $\log(I-\Delta)$ and $\mathcal{L}$. Therefore~$\mathcal{L}$ satisfies the same logarithmic Sobolev inequalities as~$\log(I-\Delta)$. Finally, we use~\Cref{re:logtohyper} and~\Cref{re:logtostronghyper} to prove that~$\mathcal{L}$ is strongly hypercontractive.

\begin{Theorem} \label{re:strongzero}
     Let~$\mathcal{L}$ be a L\'evy operator with a kernel~$J$ satisfying~$J(x,y)\geq J_{\log}(x-y)$ for~$|x-y|<1$. Then, for all~$p\geq 2$,
    \begin{equation}
        \ent(|f|^p) \leq \big(C(p) + \|J_{\log}\|_{L^1(\RR^N \setminus B_1)}\big) \|f\|_p^p + \frac{N}{2} \mathcal{E}_{\mathcal{L}}(|f|^{p-2}f,f)
    \end{equation}
    holds for all $f \in L^1(\RR^N) \cap L^p(\RR^N) \cap H_\mathcal{L}(\RR^N)$ and $C(p)$ as in~\Cref{re:logdellog}.
    
    If, furthermore,~$J(x,y)\geq J_{\log}(x-y)$ for all~$x,y \in \RR^N$, then
        \begin{equation}
        \ent(|f|^p) \leq C(p) \|f\|_p^p + \frac{N}{2} \mathcal{E}_{\mathcal{L}}(|f|^{p-2}f,f).
    \end{equation}
    Hence, in both cases~$\mathcal{L}$ is strongly~$p$--hypercontractive for all~$p > 1$ and, for~$u_0$ in~$L^p(\RR^N)$, the solution~$u(t)$ to~\eqref{problem-L} gets into every~$L^q(\RR^N)$, with~$q\geq p$, for~$t = \frac{N}{2p}$.
\end{Theorem}
\begin{proof}
    Assume as always that the functions involved are nonnegative. Since
    \[
        (a^{p-1}-b^{p-1})(a-b) \geq 0\quad\text{for all }p>1\text{ and }a,b \geq 0,
    \]
    if~$J(x,y) \geq J_{\log}(x-y)$ for all~$x,y \in \RR^N$, then
    \begin{equation}
        \mathcal{E}_{\log}(u^{p-1},u) \leq \mathcal{E}_{\mathcal{L}}(u^{p-1},u).
    \end{equation}
    From this, it is trivial to check that the operator~$\mathcal{L}$ satisfies all the same $p$--logarithmic inequalities as~$\log(I-\Delta)$, with the same constants; that is, for~$f\geq 0$,
    \begin{equation}
        \ent(f^p) \leq C(p) \|f\|_p^p + \frac N2 \mathcal{E}_{\mathcal{L}}(f^{p-1},f), 
    \end{equation}
    with~$C(p)$ as in~\Cref{re:logdellog}.

    Now assume that~$J(x,y)\geq J_{\log}(x-y)$ for~$|x-y|<1$. In this case,
    \begin{equation}
        \begin{aligned}
            \mathcal{E}_{\log}(f^{p-1},f)&=\mathcal{A}+\mathcal{B},\quad\textup{where}\\
            \mathcal{A}&=\frac12\iint_{|x-y|<1}(u(x)^{p-1}-u(y)^{p-1})(u(x)-u(y)) J_{\log}(x-y),\\
            \mathcal{B}&=\frac12\iint_{|x-y|\geq1} (u(x)^{p-1}-u(y)^{p-1})(u(x)-u(y))J_{\log}(x-y).
        \end{aligned}
    \end{equation}
    For~$\mathcal{A}$, exactly as before,
    \begin{equation}
        \mathcal{A}\leq\frac12\iint_{|x-y|<1}(u(x)^{p-1}-u(y)^{p-1})(u(x)-u(y))J_{\mathcal{L}}(x,y)\leq\mathcal{E}(f^{p-1},f).
    \end{equation}
    For~$\mathcal{B}$, we use the inequality
    \[
        (a^{p-1}-b^{p-1})(a-b) \leq \left(a^{\frac p2} - b^{\frac p2}\right)^2,
    \]
    valid for all~$p > 1$ and~$a,b \geq 0$, to obtain that
    \begin{equation}
        \begin{aligned}
            \mathcal{B}&\leq\frac12\iint_{|x-y|\geq 1}(u(x)^{\frac p2}-u(y)^{\frac p2})^2J_{\log}(x-y)\leq\iint_{|x-y|\geq1}u(x)^{p}J_{\log}(x-y)\\
            &=\|J_{\log}\|_{L^1(\RR^N\setminus B_1)}\|u\|_p^p.
        \end{aligned}
    \end{equation}
    Therefore,
    \begin{equation}
        \mathcal{E}_{\log} (u^{p-1}, u) \leq \mathcal{E}_{\mathcal{L}} (u^{p-1}, u) + \|J_{\log}\|_{L^1(\RR^N \setminus B_1)} \|u\|_p^p,
    \end{equation}
    and we obtain the corresponding~$p$--logarithmic inequalities for this operator.
    Thus, thanks to~\Cref{re:logtohyper}, the operator~$\mathcal{L}$ is~$p$--hypercontractive for all~$p\geq 2$. Furthermore, it is obvious from~\Cref{re:logtostronghyper}  that the operator is strongly~$p$--hypercontractive for all~$p\geq 2$, and for initial data~$u_0 \in L^p(\RR^N)$, solutions~$u=u(t)$ to~\eqref{problem-L} get into every~$L^q$, with~$q\geq2$, for~$t = \frac{N}{2p}$. To extend the hypercontractivity for~$p \in (1,2)$, simply apply the duality result~\Cref{re:dualityresult}.
\end{proof}

\begin{Remark}
    The coefficient~$\|J_{\log}\|_{L^1(\RR^N \setminus B_1)}$ accompanying~$C(p)$ is of no importance, since for both~$C(p)$ and~$D(p)$ what matters is how they behave when~$p \to \infty$.
\end{Remark}
\begin{Remark}
    Recall that~$C(p) \sim p - \log(p)$ for~$p\gg1$, which yields
    \begin{equation}
        \int_p^{\infty} \frac{C(\alpha)}{\alpha^2} \, d\alpha = \infty.
    \end{equation}
    This implies that solutions do not necessarily get bounded for~$t = \frac{N}{2p}$, which is the time at which they are in every~$L^q$ for~$p \leq q < \infty$. This was expected, since solutions of problem~\eqref{problem-L} associated to~$\log(I-\Delta)$ do not get bounded for~$t = \frac{N}{2p}$ either; they get bounded for every~$t> \frac{N}{2p}$.
\end{Remark}

Let us finish this section by showing that~\Cref{re:strongzero} is optimal, in the sense that~$\mathcal{L}$ cannot be supercontractive, and thus neither ultracontractive, if we assume an upper bound of the form 
\begin{equation} \label{eq:upperlogbound}
    J(x,y) \leq c|x-y|^{-N}.
\end{equation}

\begin{Theorem} \label{re:notsuper}
    Let~$\mathcal{L}$ be a L\'evy operator with a kernel~$J$ satisfying~$J(x,y)\leq c|x-y|^{-N}$ for~$|x-y|<1$ and~$c>0$. Then~$\mathcal{L}$ cannot be supercontractive.
\end{Theorem}
\begin{proof}
    From the upper bound~\eqref{eq:upperlogbound} we obtain, just like before, 
    \begin{equation} \label{eq:energycomparison}
        \mathcal{E}_{\mathcal{L}} (|f|^{p-1}f, f) \leq \widetilde{c} \left( \mathcal{E}_{\log} (|f|^{p-1}f, f) + \|f\|_p^p \right)
    \end{equation}
    for some~$\widetilde{c}>0$. If~$\mathcal{L}$ were supercontractive, it would satisfy a logarithmic Sobolev inequality of type~\eqref{eq:logdelsuper}, with a free parameter~$\varepsilon>0$, thanks to~\Cref{re:supercontractivity}. But then,~\eqref{eq:energycomparison} would imply that~$\log(I-\Delta)$ also satisfies this same logarithmic Sobolev inequality. Therefore, again by~\Cref{re:supercontractivity},~$\log(I-\Delta)$ would be supercontractive, which is a contradiction. 
\end{proof}

We have shown that the gradual improvement in integrability that appears for solutions of~\eqref{problem-L} is not exclusive of~$\mathcal{L} = \log(I-\Delta)$: any operator~$\mathcal{L}$ of the form~\eqref{general-L} whose kernel is comparable to the one of~$\log(I-\Delta)$ for small interactions satisfies the same strong hypercontractivity. However, for~$\log(I-\Delta)$, solutions get bounded immediately after getting into every $L^q$ for $q$ finite. In the next section we show that the same holds assuming that the operator~$\mathcal{L}$ is also translation invariant.

\section{Translation-invariant operators} \label{se:convolution}

In the special case where~$\mathcal{L}$ is a translation-invariant Lévy operator, i.e., the kernel~$J$ defining it through~\eqref{general-L} has the form~$J(x,y)=\widetilde{J}(x-y)$, solutions to~$\partial_t u + \mathcal{L}u = 0$ can be written as a convolution with its fundamental solution~$\mathcal{H}_t$,
$$
u(t)=\mathcal{H}_t*f,\qquad \widehat{\mathcal{H}_t}(\xi)=e^{-m(\xi)t},
$$
where the multiplier~$m$ is given by the L\'evy-Khintchine formula
\begin{equation}\label{eq:Levy.Khintchine}
    m(\xi)=\int_{\mathbb{R}^N}(1-\cos(z \cdot \xi))\widetilde{J}(z)\,dz,
\end{equation}
which is well defined for Lévy kernels. Furthermore, we know that there exist~$C_1, C_2 > 0$ such that
\begin{equation}
    C_1 \min\{1, |\xi|^2\} \leq m(\xi) \leq C_2 \max\{1, |\xi|^2\}.
\end{equation}
In general,~$\mathcal{H}_t$ may not be a function, but it is a probability measure for all~$t>0$ thanks to~$\mathcal{L}$ being a Lévy operator. For instance, the biharmonic Laplacian~$\Delta^2$ is not a Lévy operator, and the fundamental solution of~$\partial_t u +\Delta^2 u=0$ is not a probability measure. Additionally, the Fourier multiplier of the biharmonic Laplacian grows polynomially, which is too much growth for gradual improvement to hold.  For~$m$ to satisfy L\'evy--Khintchine's formula, and therefore to correspond to a L\'evy operator, it must verify $m(\xi) = -\psi(-\xi)$ for $\psi$ the characteristic exponent of an infinitely divisible distribution; see, for instance,~\cite{SatoLevy}.

Let us focus first on the fundamental solution. Observe that~$m(0) = 0$, so that
\[ 
    \displaystyle\int_{\mathbb{R}^N}\mathcal{H}_t = e^{-tm(0)} = 1,
\]
from where conservation of mass follows. Then, using Hausdorff-Young's inequality, for~$r\geq 2$ we get
\begin{equation}
    \|\mathcal{H}_t\|_r \leq \|e^{-tm}\|_{r'} = \|e^{-m}\|_{tr'}^t.
\end{equation}
This will allow us to obtain further integrability properties for the heat kernel.

\begin{Proposition} \label{re:fourierimprovement}
    Let~$1 \leq p < \infty$. If~$e^{-m} \in L^p(\RR^N)$, then~$\mathcal{H}_t \in L^1(\RR^N) \cap L^{\frac{p}{p-t}}(\RR^N)$ for~$t \in [p/2,p)$. In particular,~$\mathcal{H}_t \in L^1(\RR^N) \cap L^{\infty}(\RR^N)$ for~$t \geq p$.
\end{Proposition}

\begin{proof}
    Choosing~$r = \frac{p}{p-t}$ we get that~$r'=p/t$, and thus~$\mathcal{H}_t \in L^{\frac{p}{p-t}}(\RR^N)$ thanks to Hausdorff-Young's inequality, under the constraint that~$t \geq p/2$. Furthermore,~$\mathcal{H}_t$ is not only a probability measure, but also a function in~$L^1$. The~$L^\infty$ bound is deduced by choosing~$t=p$. For the bounds when~$t > p$, we simply use that~$\mathcal{H}_t$ is the solution to the PDE and the results from the~\hyperref[se:existence]{Appendix}.
\end{proof}
\Cref{re:fourierimprovement} does not give any information about the fundamental solution for small times, which implies that we cannot deduce from it how $\mathcal{H}_t$ gets into~$L^q$ for~$q\in (1,2)$. However,  any improvement in the integrability of the fundamental solution yields eventual ultracontractivity.

\begin{Corollary} \label{re:translationequivalence}
    If~$\mathcal{H}_{t_0} \in L^p(\RR^N)$ for some $p\in(1,\infty]$ at some time $t_0>0$, then~$\mathcal{L}$ is eventually ultracontractive.
\end{Corollary}
\begin{proof}
    If~$\mathcal{H}_{t_0} \in L^p(\RR^N)$, then~$\mathcal{H}_{t_0} \in L^1(\RR^N) \cap L^p(\RR^N)$. Through interpolation, $\mathcal{H}_{t_0} \in L^q(\RR^N)$ for some $q \in (1,2]$. Then, by Hausdorff-Young's inequality,~$e^{-t_0 m} \in L^{q'}(\RR^N)$ and the result follows from~\Cref{re:fourierimprovement} and Young's inequality.
\end{proof}
The reciprocal statement is also true.
\begin{Proposition} \label{re:translationequivalence2}
    If~$\mathcal{L}$ is eventually ultracontractive for some $p>1$, i.e., there exists some $t_0$ such that
    \begin{equation} \label{eq:ultracondition}
        \|H_{t_0} * u_0\|_\infty \leq C\|u_0\|_p
    \end{equation}
    for some $C>0$, then $H_{t_0} \in L^{p'}(\RR^N)$.
\end{Proposition}
\begin{proof}
    Define $\tau_z f(x) = f(z-x)$. Then, condition~\eqref{eq:ultracondition} implies that
    \begin{equation}
        |\langle H_{t_0}, \tau_0 f \rangle| \leq C \|\tau_0 f\|_p
    \end{equation}
    for every $f \in L^p(\RR^N)$, since $\tau_0$ is an isometry in $L^p(\RR^N)$. Hence, renaming $g = \tau_0 f$, and taking the supremum of $g \in L^p(\RR^N)$ such that $\|g\|_p = 1$, we conclude what we wanted.
\end{proof}

If~$m$ grows logarithmically at infinity, then~$e^{-m}$ clearly belongs to~$L^p$ for sufficiently large~$p$. The consequences of this are summarized in the following result, where we assume that~$m(\xi) \geq c \log(|\xi|^2)$ for~$|\xi|$ large; we may take~$c=1$ without loss of generality.

\begin{Proposition}\label{prop:in.Bessel}
    Let~$\mathcal{L}$ be a translation-invariant Lévy operator with multiplier~$m$ satisfying~$m(\xi) \geq \log(|\xi|^2)$ for~$|\xi| \gg 1$. Let~$u$ be the solution to problem~\eqref{problem-L} with~$u_0 \in L^p(\RR^N)$ for $p\geq 1$.
    \begin{itemize}
        \item If $1 \leq p < 2$,  $u(t) \in L^{\frac{(N + \varepsilon) p}{(N + \varepsilon) - 2pt}}(\RR^N)$ for $\frac {(N + \varepsilon)}{4} \leq t \leq \frac N2$. In particular,~$u(t) \in L^\infty(\RR^N)$ for~$t > \frac {N}{2p}$.
        \item  If $p=2$, then~${u(t)\in L^{2}_{2t}(\RR^N)}$ for every~$t \geq 0$. In particular,~$u(t) \in L^\infty(\RR^N)$ for $t > \frac {N}{4}$.
        \item If $p > 2$, then~$u(t) \in L^\infty(\RR^N)$ for~$t > \frac {N}{4}$.
    \end{itemize}
\end{Proposition}
\begin{proof}
    We know that the multiplier~$m$ is locally bounded. If there exists~$R>0$ such that~$m(\xi) \geq \log(|\xi|^2)$ for~$|\xi| \geq R$, then~$e^{-m(\xi)} \leq |\xi|^{-2}$. This implies that for any $\varepsilon>0$, $e^{-m(\xi)} \in L^{(N + \varepsilon)/2}(\RR^N)$. Then, the result for $p\neq 2$ follows by~\Cref{re:fourierimprovement} and Young's inequality.
    If $u_0 \in L^2(\RR^N)$, we use that~$\widehat{u}(\xi, t) = e^{-t m(\xi)} \widehat{u}_0(\xi)$ and conclude that~$u(t) \in L^2_{2t}(\RR^N)$ since~$(1+|\xi|^2)^t e^{-t m(\xi)}$ is bounded.
\end{proof}

The previous hypotheses were given in terms of the multiplier~$m$. We would also like to give a condition in terms of the kernel~$\widetilde J$ of the operator. 
Let~$\ell(|z|) = |z|^{N}\widetilde{J}(z)$. Define
\begin{gather}\label{eq:def.psi1.psi2}
    \psi_1^\ell(r) = \int_r^1 \frac{\ell(s)}{s} \, ds, \qquad \psi_2^\ell(r) = \frac{1}{r^2}\int_0^r s\ell(s) \, ds.
\end{gather}
Recall that integrable kernels cannot satisfy any kind of hypercontractivity. Therefore, we only consider non-integrable kernels; thus,~$\lim\limits_{r \to 0} \psi_1^\ell(r) = \infty$. We have that the multiplier~$m(\xi)$ satisfies
\begin{equation} \label{eq:multiplierupper}
    m(\xi)\leq c_1\psi_1^\ell\left(\frac{1}{|\xi|}\right)+c_2\psi_2^\ell\left(\frac{1}{|\xi|}\right)\quad\text{for }|\xi|>1.
\end{equation}

On the contrary, under the additional hypotheses that~$\ell(r) \geq \beta(r)$ in~$r \in (0,1)$, for~$\beta$ a positive function that varies regularly at zero with index~$\rho\in (-2,0]$, that is,
\begin{equation}
    \lim_{r\to 0^+} \frac{\beta(\lambda r)}{\beta(r)} = \lambda^\rho\quad\text{for every }\lambda>0,
\end{equation}
we have the lower estimate
\begin{equation} \label{eq:lowerboundregularlyvarying}
    m(\xi) \geq c_3 \psi_1^\beta\left(\frac{1}{|\xi|}\right)\quad\text{for } |\xi| > 1
\end{equation}
for some $c_3>0$ depending on $N$ and $\widetilde{J}$; see~\cite{KassmannMimicaIntrinsicScaling}, and also~\cite{ArturoCristinaRegularizingEffect}. In the particular case in which $\ell\ge\nu$ near the origin for some positive constant, i.e.,  if $\widetilde J(z)\ge \nu|z|^{-N}$
for $|z|<1$, then
$$
m(\xi)\ge \nu\int_{|y|<1}\frac{1-\cos(y\cdot\xi)}{|y|^N}\,dy=
\nu\int_{|z|<|\xi|}\frac{1-\cos(z_1)}{|z|^N}\,dz=\nu\left( \omega_{N-1} \log|\xi|+O(1)\right),
$$
for $|\xi|\gg1$.

We found what we wanted: a hypercontractivity result in terms of the kernel.
\begin{Corollary} \label{re:translationeventualultra}
    Let~$\mathcal{L}$ be a translation-invariant Lévy operator. If~$\widetilde{J}(z) \geq \nu|z|^{-N}$ for~$|z|<1$ and $\nu>0$, the solution~$u$ of problem~\eqref{problem-L} satisfies the conclusions of~\Cref{prop:in.Bessel} with times depending on~$\nu$ and~$N$.
\end{Corollary}

The previous results relied on studying the fundamental solution through its Fourier transform, and they have two clear drawbacks, already apparent in~\Cref{prop:in.Bessel}: 
\begin{itemize}
    \item The first one is that we obtain no information for times $t\leq \frac N4$ (unless $p=2$). 
    \item The second one is that we do not find the optimal times at which solutions get bounded for~$p>2$.
\end{itemize}
We propose a different approach that does not rely on the Fourier transform, and thus avoids these issues. This alternative method also has the advantage of yielding the optimal times after which solutions get bounded.

Assume that there exists~$c>0$ such that~$\widetilde{J} \geq c \widetilde{J}_{\log}$, at least for small interactions. With the change of variables~$t' = t/c$, we get that if~$u(t)$ verifies~$\partial_t u + \mathcal{L}u = 0$, then~$u(t')$ verifies~$\partial_t u + c\mathcal{L}u = 0$. Hence, without loss in generality, we take $c=1$.

\begin{Theorem} \label{re:commutationtheorem}
    Let~$u$ be the solution to~\eqref{problem-L} with~$\mathcal{L}$ a translation-invariant Lévy operator and initial data~$u_0 \in L^2(\RR^N) \cap L^{p}(\RR^N)$ for some $p\geq 1$. Assume that~$\widetilde{J}(z) \geq \widetilde{J}_{\log}(z)$ for $|z|<1$. Then,~$u(t) \in L_{2t}^2(\RR^N) \cap L_{2t}^p(\RR^N)$ for every~$t>0$. If, moreover,~$\widetilde{J}(z) \geq \widetilde{J}_{\log}(z)$ for all $z \in \RR^N$, then
    \begin{equation} \label{eq:Lpestimate}
        \|u(t)\|_{L_{2t}^p(\RR^N)} \leq \|u_0\|_p.
    \end{equation}
\end{Theorem}

\begin{proof}
    Define the spaces
    \begin{align}
        D(\mathcal{L}) = \{f\in L^2(\RR^N)\,:\,\mathcal{L}f \in L^2(\RR^N)\}, \qquad D(\mathcal{L}^2) = \{f\in D(\mathcal{L})\,:\,\mathcal{L}f \in D(\mathcal{L})\},
    \end{align}
    and the operators
    \begin{align}
        \mathcal{W}^{\textup{in}}f(x) &= \text{P.V.}\int_{B_1(x)} (f(x) - f(x+z)) (\widetilde{J}(z)-\widetilde{J}_{\log}(z)) \,dz, \\
        \mathcal{W}^{\textup{out}} f(x) &= \int_{\RR^N \setminus B_1(x)} (f(x) - f(x+z)) (\widetilde{J}(z)-\widetilde{J}_{\log}(z)) \,dz. 
    \end{align}
    Observe that the second operator is bounded and linear in $D(\mathcal{L}^2)$. Let $u_0\in L^p(\RR^N) \cap D(\mathcal{L}^2)$, and let~$w \in \cont^1([0, \infty]; D(\mathcal{L}^2))$ be the unique solution of the ODE
    \begin{equation}
        \partial_t w = -\mathcal{W}^{\textup{out}}w\quad\textup{in }\RR^N\times(0,\infty),\qquad
            w(\cdot,0)=u_0\quad\textup{in }\RR^N,
    \end{equation}
    whose existence is guaranteed by Picard-Lindelöf's theorem. Since $\mathcal{W}_{in}$ is a maximal monotone operator, it is the infinitesimal generator of a solution semigroup $T(t)$, see the~\hyperref[se:existence]{Appendix}. Let $v(t) = T(t)w(t)$, and note that it satisfies the equation
    \begin{equation}
        \partial_t v = -\mathcal{W}^{\textup{in}}v - \mathcal{W}^{\textup{out}} v = -\big(\mathcal{L} - \log(I-\Delta) \big)v,
    \end{equation}
    with $v(\cdot,0)=u_0$ (this is proved using that translation-invariant operators commute). Finally, define~$u(t) = (I-\Delta)^{-t}v(t)$. Since~$v(t)\in L^2(\mathbb{R}^N)\cap L^p(\mathbb{R}^N)$ for all~$t>0$, then~$u(t)\in L^2_{2t}(\mathbb{R}^N) \cap L^p_{2t}(\mathbb{R}^N)$. On the other hand, using again that translation-invariant operators commute, we find that~$u$ satisfies problem~\eqref{problem-L} with initial data~$u_0$. It is straightforward to check that it is the unique solution: just use the theory of~\Cref{re:HilleYosida} with $\mathcal{L} = \mathcal{W}^{\textup{in}}$, adding a right-hand side $\mathcal{W}^{\textup{out}} u$, which can be managed be means of Gronwall's inequality. Then, since $\mathcal{L}$ is self-adjoint, we can use a standard approximation argument to prove that we can  take $u_0 \in L^2(\RR^N) \cap L^p(\RR^N)$;~see~\cite[Theorem~7.7]{Brezisbook}. 
    
    Finally, if~$\widetilde{J} \geq \widetilde{J}_{\log}$ in the whole $\RR^N$, then it is easy to see that $\|v(t)\|_p \leq \|u_0\|_p$, and thus 
    \begin{equation} \tag*{ }
        \|u(t)\|_{L_{2t}^p} = \|v(t)\|_p \leq \|u_0\|_p. \qedhere
    \end{equation}
\end{proof}

Thanks to the previous result and the embeddings used in~\Cref{re:solutionforlog}, we can obtain the optimal time threshold after which solutions get bounded, which is  $t>N/2p$. Observe that the result is valid for anisotropic kernels $\widetilde{J}$.

Finally, we address the question of whether $\widetilde{J}(z) \sim |z|^{-N} $ for~$|z| \sim 0$ constitutes the threshold for the gradual smoothing phenomenon. The answer is affirmative: any radial perturbation that increases the singularity of the kernel renders the operator ultracontractive, whereas any radial perturbation that decreases it prevents any improvement of the fundamental solution.

\begin{Theorem} \label{re:summaryresult}
    Assume that the kernel defining the operator~$\mathcal{L}$ satisfies~$\widetilde{J}(z) = \ell(|z|) |z|^{-N}$ for~$|z| \le1$, and that~$\lim\limits_{r \to 0} \ell(r)$ exists. Then:
    \begin{itemize}
        \item[(i)] If~$\lim\limits_{r \to 0} \ell(r) = \infty$, the operator is ultracontractive. Moreover, solutions belong to~$\cont^\infty(\RR^N)$ for any~$t>0$.
        \item[(ii)] If~$\lim\limits_{r \to 0} \ell(r) = c>0$, the operator is strongly $p$--hypercontractive for every $p>1$, and also eventually ultracontractive. Solutions improve in differentiability right after getting bounded.
        \item[(iii)] If~$\lim\limits_{r \to 0} \ell(r) = 0$, the operator is not eventually ultracontractive for any $p>1$. In particular, the fundamental solution does not belong to~$L^q$ for any~$q>1$. 
    \end{itemize}
\end{Theorem}
\begin{proof}
    (i) Ultracontractivity in this case was established in~\cite[Corollary 4.5]{ArturoCristinaRegularizingEffect}. As noticed there, the condition $\lim_{r\to 0}\ell(r)=\infty$ implies that $\ell$ is bounded from below, for $r$ sufficiently small, by a positive constant (which is a regularly varying function). Consequently, the estimate~\eqref{eq:lowerboundregularlyvarying} applies, and we have
    \begin{equation}
        \lim_{|\xi|\to \infty} \frac{m(\xi)}{\log(|\xi|)} \geq c_3 \lim_{r\to 0}\frac{\psi_1^\ell(r)}{\log(1/r)} = -c_3 \lim_{r \to 0} r (\psi_1^\ell)'(r) = c_3 \lim_{r \to 0} \ell(r) = \infty.
    \end{equation}
    Thus,  the so-called Hartman-Wintner condition,
    \[
        \lim_{|\xi|\to \infty} \frac{m(\xi)}{\log(|\xi|)} = \infty,   
    \]
    holds,  which es enough to guarantee that for any $t>0$ the fundamental solution  is radial, nonincreasing in $r=|x|$ and bounded, and belongs to~$C^\infty(\RR^N)$;  see for instance~\cite{knopova_note_2013}. 
    
    \noindent (ii) Strong hypercontractivity follows in this case from~\Cref{re:strongzero} and~\Cref{re:translationeventualultra}. \Cref{re:commutationtheorem} also gives the strong hypercontractivity, as well as the optimal condition after which solutions get bounded and the differentiability improvement.
    
    \noindent (iii) If~$\lim\limits_{r \to 0} \ell(r) = 0$, it is easy to see that
    \begin{equation}
        \lim_{r \to 0} \frac{\psi_1^\ell(r) }{\log(r)} = \lim_{r \to 0} \frac{1}{\log(r)}\int^1_{r} \frac{\ell(s)}{s} \, ds = -\lim_{r \to 0}  \ell(r) = 0.
    \end{equation}
    Hence, for any~$\varepsilon>0$ we get~$\psi_1^\ell\left( \frac{1}{r} \right) \leq \varepsilon \log(r)$ for~$r$ big enough depending on~$\varepsilon$. Then, using~\eqref{eq:multiplierupper},
    \begin{equation}
        e^{-t m(\xi)} \geq e^{-c_2 t}e^{-c_1 \varepsilon t \log(|\xi|)} = \frac{e^{-c_2 t}}{|\xi|^{c_1 \varepsilon t}}\quad\text{for } |\xi| \text{ big enough.}
    \end{equation}
    Since~$\varepsilon$ is arbitrary,~$e^{-tm}$ cannot belong to any~$L^p$, with~$p \in [1,\infty)$ for any~$t>0$. Thanks to Hausdorff-Young's inequality, for~$p\geq 2$,
    \begin{equation}
        \|e^{-t m}\|_p \leq \|\mathcal{H}_t\|_{p'}.
    \end{equation}
    Thus,~$\mathcal{H}_t$ does not belong to any~$L^q$, with~$q \in (1,2]$, for any~$t>0$. If~$\mathcal{H}_t$ enters some space~$L^q$ with~$q>2$ for some big enough time, it also belongs to~$L^1$ from that time onward. Then, by interpolation, we could deduce that it also belonged to all~$L^q$ spaces with~$q \in (1,2]$. Hence,~$\mathcal{H}_t$ does not get into any~$L^q$ for any~$q>1$. Assume that $\mathcal{L}$ was eventually ultracontractive for some $p>1$. Then, due to~\Cref{re:translationequivalence2}, we find a contradiction.  
\end{proof}

There is a small subtlety in Theorem~\ref{re:summaryresult}(iii): we show that the fundamental solution does not gain integrability, but we cannot a priori rule out this phenomenon for other solutions with $L^p$ initial data. To this end, we require an analogue of~\Cref{re:translationequivalence2} valid in $L^r$, for arbitrary $r$, rather than only in $L^\infty$. In other words, we need to ``invert'' weak Young's inequality.

\begin{Lemma}[Optimality of weak Young's inequality] \label{re:optimalweak}
    Let $p, r \in \RR$ such that $1 \leq p < r <\infty$, and let $G \geq 0$ be a measurable function satisfying 
    \begin{equation} \label{eq:weakyoungstart}
        \|G*f\|_r \leq C \|f\|_p
    \end{equation}
    for $C>0$ and any $f \in L^p(\RR^N)$. For $\lambda \geq 0$, define the sets
    \begin{equation}
        E_\lambda = \{x \in \RR^N : G(x) > \lambda \}.
    \end{equation}
    Assume for any $\lambda \geq 0$ that 
    \begin{equation} \label{eq:sumhypothesis}
        \mu(E_\lambda + E_\lambda) \leq c \mu(E_\lambda)
    \end{equation}
    for some $c>0$ independent of $\lambda$. Then $G \in L^{q, \infty}(\RR^N)$ for $q$ satisfying
    \begin{equation}
        1 + \frac1r = \frac1p + \frac1q.
    \end{equation}
\end{Lemma}
\begin{proof}
    Let $f(x) = \chi_{E_\lambda}(x)$. From~\eqref{eq:weakyoungstart}, we get that
    \begin{equation} \label{eq:startingweakproof}
        \|G*\chi_{E_\lambda}\|_r \leq C \|\chi_{E_\lambda}\|_p = C \mu(E_\lambda)^{1/p}.
    \end{equation}
    On the other hand, 
    \begin{equation}
        G*\chi_{E_\lambda}(x)\geq \int_{E_\lambda} G(y)\chi_{E_\lambda}(x-y)\, dy 
        \geq \lambda  \big(\chi_{E_\lambda}*\chi_{E_\lambda}\big)(x).
    \end{equation}
    Now, notice that 
    \begin{equation}
        \mu^2(E_\lambda) = \|\chi_{E_\lambda}*\chi_{E_\lambda}\|_1 = \int_{E_\lambda + E_\lambda} \chi_{E_\lambda}*\chi_{E_\lambda} \leq \|\chi_{E_\lambda}*\chi_{E_\lambda}\|_r \, \mu(E_\lambda + E_\lambda)^{1-\frac1r}.
    \end{equation}
    Hence, using~\eqref{eq:sumhypothesis},
    \begin{equation}
        \|G*\chi_{E_\lambda}(x)\|_r \geq \lambda\, \frac{\mu^2(E_\lambda)}{\mu(E_\lambda + E_\lambda)^{1-\frac1r}}\geq \widetilde{c} \lambda \mu(E_\lambda)^{1+\frac1r}.
    \end{equation}
    This, together with~\eqref{eq:startingweakproof} implies that $\lambda \mu(E_\lambda)^{1/q}$ is bounded independently of $\lambda\geq0$, and thus we conclude the result. 
\end{proof}

Hypothesis \eqref{eq:sumhypothesis} is trivially satisfied if $E_\lambda$ is a convex set. In particular, if $\mathcal{H}_t$ is radial and nonincreasing in $r=|x|$ for each $t>0$, then the sets $\{x \in \RR^N: H_t(x) > \lambda\}$ are balls and thus satisfy the condition. In order to ensure that~$\mathcal{H}_t$ satisfies these requirements, we introduce the following definition, which allows us to state a sufficient condition.

\begin{Definition}
    A measure $\mu(dx)$ on $\RR^N$ is \emph{isotropic} if there exists a function $f(r)$ $(0<r<\infty)$ such that $\mu(dx) = f(|x|)\, dx$ for $x\neq 0$. A measure $\mu(dx)$ on $\RR^N$ is \emph{isotropic unimodal} if $f(r)$ is nonincreasing ($\mu(\{0\}$ may be positive).
\end{Definition}

It turns out that the fundamental solution~$\mathcal{H}_t$ is isotropic unimodal if and only if the kernel $\widetilde{J}(z)$ is isotropic unimodal; see~\cite{watanabe_unimodal_1983}.  Furthermore,~$\widetilde{J}(z)$ is integrable if and only if~$\int_{\{0\}}\mathcal{H}_t\, dx > 0$ for $t>0$. With this in mind, we conclude the result below.
\begin{Corollary}
    Let the kernel~$\widetilde{J}$ defining the operator~$\mathcal{L}$ be radial and nonincreasing in $r=|x|$, and assume that $\lim_{r \to \infty}\ell(r) = 0$. Then, 
    \begin{enumerate}
        \item[(i)] If the kernel is integrable, solutions are only as regular as the initial data. 
        \item[(ii)] If the kernel is not integrable, then solutions do not improve in integrability.
    \end{enumerate}
    In any case, the operator is not $p$--hypercontractive for any $p \geq 1$.
\end{Corollary}
\begin{proof}
    (i) If the kernel is integrable, the solution to~\eqref{problem-L} is given by~$u(t) = e^{-t} u_0 + v(t)*u_0$ for some function~$v$ (the regular part of the heat kernel); see for instance~\cite{ChasseigneChavesRossi2007}. Therefore,~$u(t)$ is only as good as the initial datum is.

     \noindent (ii) If the kernel is not integrable, then we can guarantee that the fundamental solution~$\mathcal{H}_t$ is radial and nonincreasing in $|x|$ for all $t>0$, since it has no Dirac delta at the origin. Thus, if solutions starting from initial data in~$L^p$ with~$p \geq 1$ improved in integrability,~\Cref{re:optimalweak} would force~$\mathcal{H}_t$ to belong to~$L^q(\RR^N)$ for some $q>1$. However, this contradicts~\Cref{re:summaryresult}.
\end{proof}
Hence, not only is $\widetilde{J}(z) \sim |z|^{-N}$ for $|z|\ll 1$ the threshold for eventual ultracontractivity, but also for any kind of hypercontractivity. Observe that this previous result even works for~${p=1}$, which is the most problematic case throughout the paper.

\section{Comments and possible extensions} \label{se:extensions}

Let us highlight what we regard as the most significant open question raised by this work: Are strong hypercontractivity and eventual ultracontractivity equivalent? Alternatively, can we find operators of the form~\eqref{general-L} whose solutions to~\eqref{problem-L} enter every space~$L^q$ for~$q \in [1,\infty)$ in finite time, yet never reach~$L^\infty$? According to our results, if such operators exist, we do not expect them to be translation-invariant. The study of how regular solutions can become with only ellipticity assumptions in the spirit of~\Cref{re:strongzero} and~\Cref{re:notsuper}, would correspond, in this setting, to the parabolic De Giorgi--Nash--Moser theory for local divergence-form operators. However, without additional regularity hypotheses on the kernel, at most $L^\infty$ is expected in our nonlocal framework, and even this may require some additional assumptions.

It is well established in the literature, under mild conditions, that the standard instantaneous ultracontractivity and Nash-type inequalities are equivalent; see~\cite{CKSNash,TomisakiNash,GrigorNash,CoulhonNash}. As shown in~\cite{TomisakiNash,CoulhonNash}, a Nash-type inequality of the form
\begin{equation}
    \varTheta(\|f\|_2) \leq \overline{\mathcal{E}}(f) \quad \text{ for } \|f\|_1 \leq 1
\end{equation}
implies ultracontractivity whenever
\begin{equation} \label{eq:nashintegralcond}
    \int^\infty \frac{1}{\varTheta(s)}\,ds < \infty.
\end{equation}

As noted in~\Cref{subse:precedents}, it was shown in~\cite{GentilMaheuxNash} that the operator~$\log(I-\Delta)$ satisfies a Nash-type inequality of the form
\begin{equation} \label{eq:lognash}
    A_\rho \|f\|_2^2 \log\left(1 + B_{\rho,N} \|f\|_2^{4/N}\right) \leq \overline{\mathcal{E}}_{\log}(f), \quad \|f\|_1 \leq 1,
\end{equation}
where~$A_\rho,B_{\rho,N}>0$ depend on a constant~$\rho>1$ and the dimension~$N$.
Observe that the left-hand side of~\eqref{eq:lognash} does not satisfy the integral condition~\eqref{eq:nashintegralcond}, as expected. To the best of our knowledge, no connection has been established in the literature between eventual ultracontractivity and Nash inequalities of the form~\eqref{eq:lognash}. It remains open whether such an equivalence exists, and whether it is related to our approach based on a family of~$p$--Nash inequalities.

Another question, of a more technical nature, is whether there exists a method to prove the~$p$--logarithmic Sobolev inequalities~\eqref{eq:log-sob-ineq} without relying on hypercontractivity. Such an approach would be particularly valuable when working on bounded domains, where the Fourier transform---heavily used in the current argument---is not available. 

Comparing the Dirichlet problem with its counterpart posed on the whole space shows that the  solution semigroup for the equation $\partial_t u+\mathcal{L}u=0$  with zero exterior data is strongly hypercontractive when the kernel $J$ defining the operator satisfies $J(x,y)\sim |x-y|^{-N}$ for $x\sim y$. However, the presence of zero exterior data introduces an additional diffusive effect that may enhance integrability more rapidly than in the whole-space case. In principle, this phenomenon could lead to instantaneous smoothing, yielding supercontractivity or perhaps even ultracontractivity. Can such behaviour be ruled out? Is a genuinely gradual smoothing still compatible with the boundary effect? These questions deserve further investigation.

A final and rather intriguing observation emerges when the problem is analysed as a critical case of fractional Porous Media equations as
$$
    \partial_t u+(-\Delta)^{\sigma/2}u^m=0,
$$ 
where~$m = m_c = \frac{(N - \sigma)_+}{N}$ is the critical exponent. Instantaneous ultracontractivity (known as~$L^1$--$L^\infty$ smoothing effect in this setting) appears when~$m>m_c$, and there is lack of instantaneous regularization when~$m \leq m_c$. What makes this particularly noteworthy is that for operators of zero or $0^+$-order, such as~$\log(I - \Delta)$, we may think of~$\sigma$ as~$0$, and consequently~$m_c = 1$. In this case, the linear scenario becomes critical.

In the range~$m\le m_c$ there is still ultracontractivity if the initial datum belongs to~$L^p$ with~$p\ge p^*(m)=\frac{N}{\sigma}(1-m)$, but not in general below this critical value~\cite{dePablo-Quiros-Rodriguez-Vazquez-CPAM-2012,Boudaoud-dePablo-Quiros-Preprint-2026}, since~$p^*(m)=\infty$ if~$m\in(0,1)$. For solutions of Porous Media-type equations
\[
    \partial_t u + \mathcal{L} u^m = 0
\]
with more general kernels analogous ultracontractivity results are expected. These observations suggest that the equation
\[
    \partial_t u + \log(I - \Delta) u^m = 0
\]
deserves further investigation in both the slow diffusion regime ($m > 1$) and the fast diffusion regime ($m < 1$). For~$m > 1$, ultracontractivity is anticipated, see~\cite{BonforteEndalSmoothing}; whereas for~$m < 1$ no improvement in integrability is expected. We plan to address both regimes in future work, and it is foreseeable that the analysis for~$m < 1$ will present significantly greater challenges than for~$m > 1$.

\section*{Appendix: Existence of solutions}
\label{se:existence}

\setcounter{equation}{0} 
\renewcommand{\theequation}{A.\arabic{equation}}

We recall the classic Hille-Yosida theorem, see~\cite{Brezisbook}.

\begin{Theorem}[Hille-Yosida for self-adjoint operators]\label{re:HilleYosida}
    Let~$H$ be a Hilbert space, and let
    \[
        {\mathcal{L}: D(\mathcal{L}) \subset H \to H}
    \]
    be a self-adjoint, maximal monotone operator. Then, for every~$u_0 \in H$ there exists a unique function
    \begin{equation} \label{eq:functionalcurve}
        u \in \cont([0,\infty); H) \cap \cont((0,\infty); D(\mathcal{L})) \cap \cont^1((0,\infty); H)
    \end{equation}
    such that
    \begin{equation}\label{eq:functionalODE}
        \dfrac{du}{dt}+\mathcal{L}u =0\quad\text{in } (0,\infty), \qquad u(0) = u_0.
    \end{equation}
    Moreover, we have
    \begin{equation}
        \|u(t)\|_H \leq \|u_0\|_H, \quad  \left\| \frac{du}{dt} (t)\right\|_H = \|\mathcal{L} u(t)\|_H \leq \frac{1}{t} \|u_0\|_H\quad\textup{for all }t>0,
    \end{equation}
    \begin{equation}
        u \in \cont^k((0, \infty) ; D(\mathcal{L}^l))\quad\textup{for all } k, l \text{ integers}.
    \end{equation}
    Additionally, if~$u_0 \in D(\mathcal{L}^k)$, then
    \begin{equation}
        u \in \cont^{k-j}([0, \infty); D(\mathcal{L}^j))\quad\textup{for all } j= 0, 1, \hdots, k.
    \end{equation}
\end{Theorem}
We can use this theorem to prove the existence of solutions to problem~\eqref{problem-L}, where~$H = L^2(\RR^N)$,~$D(\mathcal{L}) = H_\mathcal{L}(\RR^N)$, and the operator~$\mathcal{L}$ of type~\eqref{general-L} is understood via~$\langle \mathcal{L} f, g\rangle \coloneq \mathcal{E}(f,g)$ for~$f, g \in  H_\mathcal{L}(\RR^N)$. It is also worth mentioning that for this existence result to work, we do not require the kernel~$J$ to be a Lévy kernel, it is enough that it be nonnegative and symmetric.

\begin{Definition}
    We say that~$w$ is a \emph{supersolution} of~\eqref{problem-L} if~$w$ verifies~\eqref{eq:functionalcurve} and
    \begin{equation}
        \frac{dw}{dt}+\mathcal{L}w\geq 0\quad\textup{in }(0,\infty).
    \end{equation}
    It is a \emph{subsolution} if it satisfies the reverse inequality.
\end{Definition}

Observe that a solution is both a supersolution and a subsolution. Next we state a comparison result that is deduced from the standard \emph{Stampacchia’s truncation method}, and thus we omit the proof. The contractivity property expressed in the forthcoming integral formula~\eqref{eq:Tcontraction}  is usually known as \emph{$T$--contractivity}. In particular, this result implies that nonnegative initial data produce a nonnegative solution. 

\begin{Proposition}[Maximum principle]
    Let~$w,v$ be a supersolution and a subsolution of problem~\eqref{problem-L}, with initial data~$w_0$ and~$v_0$, respectively. Then
    \begin{equation} \label{eq:Tcontraction}
        \int_{\RR^N} (v(t)-w(t))_+ \leq \int_{\RR^N} (v_0 - w_0)_+\quad\textup{for all }t>0.
    \end{equation}
    Hence, if~$w_0 \geq v_0$, then~$w \geq v$.
\end{Proposition}

This maximum principle has the following consequences.
\begin{Corollary} \label{re:pboundedness}
    Let~$u$ be a solution to problem~\eqref{problem-L} with~$u_0 \in L^2(\RR^N)$. If, in addition,~$u_0 \in L^p(\RR^N)$ for some~$p \in [1,\infty]$, then~$u(t) \in L^p(\RR^N)$ and
    \begin{equation}
        \|u(t)\|_p \leq \|u_0\|_p\quad\textup{for all }t \geq 0.
    \end{equation}
    Furthermore, if~$p\geq 2$ and~$u_0 \in L^\infty(\RR^N)$, then~$u \in \cont([0,\infty); L^p(\RR^N))$; if~$1 \leq p\leq 2$ and~$u_0 \in L^1(\RR^N)$, then~$u \in \cont([0,\infty); L^p(\RR^N))$.
\end{Corollary}
\begin{proof}
    We start with the case~$p=1$. The inequality~$\|u(t)\|_1 \leq \|u_0\|_1$ is deduced trivially from~\eqref{eq:Tcontraction}. Now, observe that
    \begin{equation}
        \|u(h)-u_0\|_1 = \int_{B_R} |u(h)-u_0| + \int_{\RR^N \setminus B_R} |u(h)-u_0|.
    \end{equation}
    Thanks to the previous inequality,~$|u(h)-u_0|$ is uniformly in~$L^1(\RR^N)$, and thus there exists~$R>0$ independent of~$h$ such that
    \begin{equation}
        \int_{\RR^N \setminus B_R} |u(h)-u_0| < \frac \varepsilon2.
    \end{equation}
    Now, for this fixed~$R$, choose~$h$ so that
    \begin{equation}
        \int_{B_R} |u(h)-u_0| \leq |B_R|^{1/2}\,\|u(h)-u_0\|_2 \leq \frac \varepsilon2.
    \end{equation}
    From this, we get that~$\lim\limits_{h\to 0^+} \|u(h)-u_0\|_1 = 0$ and using the previous bound we get the continuity.

    Now we deal with the case~$p= \infty$. Again by~\eqref{eq:Tcontraction}, observe that~$\inf_{\RR^N} u_0 \leq u_0 \leq \sup_{\RR^N}u_0$, and that~$\sup_{\RR^N}u_0$ is a supersolution, and~$\inf_{\RR^N} u_0$ a subsolution. Hence,
    \begin{equation}
        \inf_{\RR^N} u_0 \leq u(t) \leq \sup_{\RR^N}u_0\quad\textup{for all }t \geq 0.
    \end{equation}
    
    Next, assume that~$2 < p < \infty$. Multiply the equation by~$u|u|^{p-2}$, that belongs to~$H_\mathcal{L}(\RR^N)$ thanks to~$p\geq2$, and integrate to obtain
    \begin{equation}
        \frac{d}{dt}\int_{\RR^N} |u|^p = -p \mathcal{E}(u|u|^{p-2} ,u) \leq 0.
    \end{equation}
    For the last inequality, we simply used that~$(a-b)(\varphi(a) - \varphi(b)) \geq 0$ for any nondecreasing function~$\varphi$. 

    The case~$1< p < 2$ is slightly more involved. Consider~$u_0 \in L^2(\RR^N) \cap L^\infty(\RR^N)$, and take an increasing function~$\varphi\in \cont^1(\RR)$ such that~$\varphi(0) = 0$. By~\Cref{re:betterfunction}, we have that~$\varphi(u) \in H_\mathcal{L}(\RR^N)$. With the same argument as before we obtain that
    \begin{equation}
        \frac{d}{dt}\int_{\RR^N} u \varphi(u) \leq 0 \implies \int_{\RR^N} u(t) \varphi(u(t)) \leq \int_{\RR^N} u_0 \varphi(u_0).
    \end{equation}
    Now we simply take~$\varphi(s)$ as a suitable approximation of~$s|s|^{p-2}$, i.e., with~$\varphi'(0)$ finite. Finally, we can avoid the~$L^\infty(\RR^N)$ hypothesis via another approximation.
    
    Continuity follows by interpolation using the~$L^2(\RR^N)$ continuity and either the~$L^\infty(\RR^N)$ or~$L^1(\RR^N)$ boundedness.
\end{proof}

\section*{Acknowledgements}

\textsc{A. de Pablo}, \textsc{F. Quir\'os} and \textsc{J. Ruiz-Cases} were supported by grants PID2020-116949GB-I00, PID2023-146931NB-I00, RED2022-134784-T, RED2024-153842-T and CEX2023-001347-S, all of them funded by MICIU/AEI/10.13039/501100011033.

\textsc{J. Ruiz-Cases} would also like to thank \textsc{X. Ros-Oton} for the fruitful discussions related to~\Cref{re:optimalweak} during a visit to Universidad de Barcelona funded by RED2024-153842-T.


\end{document}